\documentclass[11pt,a4paper]{article}

\usepackage{amsfonts}
\usepackage{amsthm}
\usepackage{amsmath}
\usepackage{amssymb}
\usepackage{amscd}
\usepackage{mathrsfs}
\usepackage[mathscr]{eucal}
\usepackage{graphicx}
\usepackage{epstopdf} 
\usepackage{pict2e}
\usepackage{epic}
\usepackage{xcolor}
\usepackage{float}
\usepackage{mathtools}
\usepackage{centernot}
\usepackage{tikz}

\usepackage{cite}
\usepackage{hyperref}
\hypersetup{colorlinks=true, urlcolor= black, linkcolor=black, citecolor=black}

\usepackage[utf8]{inputenc}
\usepackage[T1]{fontenc}
\usepackage{lmodern}
\usepackage{dsfont}
\usepackage{indentfirst}

\numberwithin{equation}{section}

\theoremstyle{plain}
\newtheorem{theorem}{Theorem}[section]
\newtheorem{proposition}{Proposition}[section]
\newtheorem{remark}{Remark}[section]

\newtheorem{corollary}{Corollary}[section]

\newtheorem{lemma}{Lemma}[section]

\newcommand{\ed}{\stackrel{\mbox{\tiny $(d)$}}{=}}

\usepackage[margin=3cm]{geometry}

\def\e{\mathbb{E}}
\def\p{\mathbb{P}}

\newcommand{\ind}{\mbox{\rm 1\hspace{-0.04in}I}}
\newcommand{\1}[1]{\mathds{1}_{\{#1\}}}

\title{Explosion speed of continuous state branching processes indexed by the Esscher transform}

\begin{document}

\author{Lo\"{\i}c Chaumont\footnote{Univ Angers, CNRS, LAREMA, SFR MATHSTIC, F-49000 Angers, France\newline 
\url{loic.chaumont@univ-angers.fr}, \url{clement.lamoureux@univ-angers.fr}}\quad\;\; Cl\'ement Lamoureux\footnotemark[1]}
\maketitle

\begin{abstract} 
A branching process $Z$ is said to be non conservative if it hits $\infty$ in a finite time with
positive probability. It is well known that this happens if and only if the branching mechanism $\varphi$ 
of $Z$ satisfies $\int_{0+}d\lambda/|\varphi(\lambda)|<\infty$. We construct on the same 
probability space a family of conservative continuous state branching processes $Z^{(\varepsilon)}$, 
$\varepsilon\ge0$, each process $Z^{(\varepsilon)}$ having 
$\varphi^{(\varepsilon)}(\lambda)=\varphi(\lambda+\varepsilon)-\varphi(\varepsilon)$ as branching 
mechanism, and such that the family 
$Z^{(\varepsilon)}$, $\varepsilon\ge0$ converges a.s.~to $Z$, as $\varepsilon\rightarrow0$. Then we 
study the speed of convergence of $Z^{(\varepsilon)}$, when $\varepsilon\rightarrow0$, referred 
to here as the explosion speed. More specifically, we characterize the functions $f$ with 
$\lim_{\varepsilon\rightarrow0} f(\varepsilon)=\infty$ and such that the first passage times 
$\sigma_\varepsilon=\inf\{t:Z^{(\varepsilon)}_t\ge f(\varepsilon)\}$ converge toward the explosion time 
of $Z$. Necessary and sufficient conditions are obtained for the weak convergence and convergence 
in $L^1$. Then we give a sufficient condition for the almost sure convergence.
\end{abstract}


\noindent {\bf Keywords} Continuous state branching process, spectrally positive L\'evy process, 
Essher transform, Lamperti transform, first passage time

\section{Introduction}

Continuous state branching processes (CSBP) have two absorbing states: 0 and $\infty$. We say that extinction 
(resp.~explosion) occurs if the state 0 (resp.~$\infty$) is attained in a finite time. Unlike the extinction 
properties which have been intensively studied in the past, the explosion of CSBP's does not seem to have 
been the subject of much research. The extinction and explosion conditions stated in Theorems \ref{5574} and 
\ref{5374} of the following section, however, reveal a sort of duality between these two phenomena.
Our searches in the existing literature on the explosion of CSBP's have only led us to very few 
articles, most of them being quite recent: Li \cite{li}, Li and Zhou \cite{lz}, Li, Foucart and Zhou 
\cite{lfz} and some references therein. 

Both articles \cite{lz} and \cite{lfz} actually consider the more general case of nonlinear CSBP's. In the 
first one it is proved that when the process is non conservative (that is explosion occurs with positive 
probability) on its event of explosion, it tends to infinity in a finite time asymptotically along some 
deterministic curve. This speed of explosion is also characterised by the speed of convergence of the process 
of first passage times toward the explosion time. The second article provides sharper results in the case 
where the branching mechanism is regularly varying. 

In the present article, we propose a quite different way of studying the explosion of a CSBP, $Z$, by 
constructing, on the same probability space, a sequence of conservative processes $(Z^{(n)})$ that converges 
to the process $Z$. Then instead of considering the speed of explosion of the process $Z$ itself, we study 
the speed of convergence of $(Z^{(n)})$ toward $Z$. This speed is characterised, on the set of explosion of 
$Z$, by the speed of convergence of the first passage times $\sigma^{(n)}_{v_n}$ of $Z^{(n)}$ above level 
$v_n$, for some sequence $(v_n)$ that tends to $\infty$. Intuitively, if $(v_n)$ tends to $\infty$ not too 
fast, then $(\sigma^{(n)}_{v_n})$ is expected to converge to the explosion time $\zeta$ of $Z$ whereas if 
it is too fast, then $(\sigma^{(n)}_{v_n})$ should tend to $\infty$. We shall study the weak convergence, 
the convergence in $L^1$ norm and the almost sure convergence. Then we shall see that in the case of weak 
convergence an unexpected phenomenon occurs: when the speed of $(v_n)$ belongs to some critical domain, 
the weak limit of $(\sigma^{(n)}_{v_n})$ is neither $\zeta$ nor $\infty$ but the law of $\zeta+c$, 
where the constant $c>0$ can be considered as some 'residual mass'. This phenomenon is specific only to 
the weak convergence. 

There are many ways of constructing a sequence of conservative branching processes $(Z^{(n)})$ that 
converges to some non conservative process $Z$, but perhaps the most natural way is to define $Z^{(n)}$
with branching mechanism $\varphi^{(n)}(\lambda):=\varphi(\varepsilon_n+\lambda)-\varphi(\varepsilon_n)$,
where $\varepsilon_n\downarrow0$ and $\varphi$ is the branching mechanism of $Z$. Then the weak convergence 
of $(Z^{(n)})$ toward $Z$ clearly holds. Moreover, it is possible to construct the sequence $(Z^{(n)})$ on 
the same probability space as $Z$. This coupling is one of the main results of the present 
paper. We obtain the sequence $(Z^{(n)})$ by first constructing an increasing sequence of 
L\'evy processes $(X^{(n)})$ such that each process $X^{(n)}$ has branching mechanism 
$\varphi^{(n)}$ defined above, so that the law of $X^{(n)}$ is an Esscher transform of the law of $X$. 
Then $Z^{(n)}$ is obtained by the well known means of the Lamperti transformation. 
This construction ensures an a.s.~uniform convergence of $X^{(n)}$ toward $X$ and an a.s.~convergence 
of $Z^{(n)}$ toward $Z$ in the Skohorod's topology. 

The next section is devoted to some basic notions on branching processes as well as the construction of 
the sequence $(Z^{(n)})$. Then we state our main results on the explosion speed of $(Z^{(n)})$ in Section 
\ref{9936} and these results are proved in Section \ref{1109}.

\section{A family of CSBP indexed by the Esscher transform}

\subsection{A brief review of CSBP's}\label{9440}

A continuous state branching process (CSBP)
$(Z_t,\,t\ge0)$ is a $[0,\infty]$-valued Markov process with probabilities $(P_x)$ satisfying the 
following property, called the branching property,
\[E_{x_1+x_2}(e^{-\lambda Z_t})=E_{x_1}(e^{-\lambda Z_t})E_{x_2}(e^{-\lambda Z_t}),\;\;\;x_1,x_2,\lambda,t\ge0.\]
This property implies that the states 0 and $\infty$ are absorbing, see \cite{gr}. Moreover there exists a 
differentiable function $u_t:[0,\infty)\rightarrow[0,\infty)$, called the Laplace exponent of $Z$, which satisfies 
\begin{equation}\label{3071}
\e_x(e^{-\lambda Z_t})=e^{-xu_t(\lambda)},\;\;\;x,\lambda,t\ge0\,.
\end{equation}
Then the Markov property yields the following semi-group property of $u_t$, 
\[u_{t+s}(\lambda)=u_t(u_s(\lambda)),\;\;\;s,t,\lambda\ge0,\]
from which we derive that $u_t$ solves the differential equation,  
\begin{equation}\label{3451}
\frac{\partial u_t}{\partial t}(\lambda)+\varphi(u_t(\lambda))=0\;\;\;;\;\;\;u_0(\lambda)=\lambda\,,
\end{equation}
where $\varphi:[0,\infty)\rightarrow\mathbb{R}$ is the Laplace exponent of a spectrally positive L\'evy 
process (spLp) that we will denote $(X_t,\,t\ge0)$, that is
\[\e(e^{-\lambda X_t})=e^{t\varphi(\lambda)},\;\;\;\lambda\ge0.\]
We denote by $\p_x$, $x\in\mathbb{R}$ a family of probability measures under which $X$ starts from 
$x$ and we set $\p_0:=\p$. The function $\varphi$ is called the  branching mechanism of $Z$. It is a 
log-convex function such that $\lim_{\lambda\rightarrow\infty}\varphi(\lambda)=\infty$ if $X$ is not a 
subordinator and $\varphi\le0$ when $X$ is a subordinator. 
According to the L\'evy-Khintchine formula, $\varphi$ can be expressed as,
\begin{equation}\label{3455}
\varphi(\lambda)=-q+a\lambda+\frac12\sigma^2\lambda^2+
\int_{(0,\infty)}\left(e^{-\lambda x}-1+\lambda x\ind_{\{x<1\}}\right)\,\pi(dx)\,,
\end{equation}
where $q\ge0$, $a\in\mathbb{R}$, $\sigma\ge0$ and $\pi$ is a L\'evy measure, that is, since $X$ has no
positive jumps, a measure on $(0,\infty)$ such that $\int_{(0,\infty)}(x^2\wedge1)\pi(dx)<\infty$. 
We can check that for all $\lambda>0$, $u_t$ is the unique solution of (\ref{3451}) and it is given for 
all $t\ge0$ by, 
\begin{equation}\label{3453}
\int_{u_t(\lambda)}^\lambda\frac{du}{\varphi(u)}=t\,.
\end{equation}

Let us now focus on the asymptotic behaviour of CSBP's. We first define the following (disjoint) sets, 
\[A_0=\left\{\lim_{t\rightarrow\infty}Z_t=0\right\}\;\;\mbox{and}\;\;
A_\infty=\left\{\lim_{t\rightarrow\infty}Z_t=\infty\right\}\,.\]
Note that each of the absorbing states 0 and $\infty$ can be attained by $Z$ in a finite or infinite 
time. Define the largest root of the branching mechanism $\varphi$ as,
\[\rho=\sup\{\lambda\ge0:\varphi(\lambda)\le0\}\,,\]
and note that $\rho=\infty$ if and only if $\varphi$ is the Laplace exponent of a subordinator. 
Then the events $A_0$ and $A_\infty$ satisfy the dichotomy described in the two following theorems, 
see \cite{gr} and Chapter 12 of \cite{ky}. 

\begin{theorem}\label{5574}
Let $Z$ be a CSBP with branching mechanism $\varphi$. Then $\p_x(A_0\cup A_\infty)=1$, for all $x\ge0$.
Moreover,
\[P_x(A_0)=e^{-x\rho}\;\;\;\mbox{and}\;\;\;P_x(A_\infty)=1-e^{-x\rho}\,.\]
\end{theorem}

\noindent In turn, each of the events $A_0$ and $A_\infty$ can be partitioned in the following way:
the event $A_0$ can be written as the union of two disjoint sets, 
$A_0=A_0^\rightarrow\cup A_0^\downarrow$, where 
\[A_0^\rightarrow=\left\{\lim_{t\rightarrow\infty}Z_t=0\;\;\mbox{and}\;\;Z_t>0\;\;
\mbox{for all $t>0$}\right\}\;\;\mbox{and}\;\;
A_0^\downarrow=\left\{Z_t=0,\;\;\mbox{for some $t>0$}\right\}\,,\]
and the event $A_\infty$ can be written as the union of two disjoint sets,  
$A_\infty=A_\infty^\rightarrow\cup A_\infty^\uparrow$, where 
\[A_\infty^\rightarrow=\left\{\lim_{t\rightarrow\infty}Z_t=\infty\;\;\mbox{and}\;\;Z_t<\infty\;\;
\mbox{for all $t>0$}\right\}\;\;\mbox{and}\;\;
A_\infty^\uparrow=\left\{Z_t=\infty,\;\;\mbox{for some $t>0$}\right\}\,.\]
Then the following dichotomy holds for each of the events $A_0$ and $A_\infty$.

\begin{theorem}\label{5374}
Let $Z$ be a CSBP with branching mechanism $\varphi$. Then  the following assertions regarding 
the event of extinction $A_0^\downarrow$ are equivalent,
\begin{itemize}
\item[$(i)$] There is $\theta>0$, such that $\displaystyle \int_\theta^\infty\frac{du}{|\varphi(u)|}<\infty$,
\item[$(ii)$] for all $x\ge0$, $\displaystyle P_x(A_0^\downarrow)=P_x(A_0)=e^{-x\rho}$ and $\rho<\infty$,
\item[$(iii)$] for all $t>0$, $u_t(\infty)<\infty$.
\end{itemize}
Similarly, for the event of explosion $A_\infty^\uparrow$, the following assertions are equivalent,
\begin{itemize}
\item[$(j)$] There is $\theta>0$, such that $\displaystyle \int_{0}^\theta\frac{du}{|\varphi(u)|}<\infty$,
\item[$(jj)$] for all $x\ge0$, $P_x(A_\infty^\uparrow)=P_x(A_\infty)=1-e^{-x\rho}$ and $\rho>0$,
\item[$(jjj)$] for all $t>0$, $u_t(0)>0$.
\end{itemize}
\end{theorem}

\subsection{Path construction of a family of CSBP's}\label{1836}

Let us first recall the Esscher transform of a spLp. This can simply be expressed from the Laplace 
exponent $\varphi$ of the process as follows: for all $\varepsilon\ge0$, the function
\[\varphi^{(\varepsilon)}(\lambda):=\varphi(\lambda+\varepsilon)-\varphi(\varepsilon),\quad\lambda\ge0,\] 
remains the Laplace exponent of a spLp. In terms of  martingale change of measure, the Esscher transform
of a spLp $X$ with Laplace exponent $\varphi$ is (the law of) a spLp $X^{(\varepsilon)}$ with Laplace exponent 
$\varphi^{(\varepsilon)}$ given by
\[\e_x(F(X^{(\varepsilon)}_s,\,s\in[0,t]))=
\e_x\left(F(X_s,\,s\in[0,t])e^{-\varepsilon X_t-t\varphi(\varepsilon)}\right)\,,\]
for all $t>0$ and all bounded, measurable functional $F$, see for instance Theorem 3.9 in \cite{ky}.\\ 

Given the Laplace exponent $\varphi$ of a spLp, through our next result, we show how to construct on the 
same probability space, an increasing family of spLp's $\{X^{(\varepsilon)},\,\varepsilon\ge0\}$ such that 
for each $\varepsilon\ge0$, $X^{(\varepsilon)}$ has Laplace exponent $\varphi^{(\varepsilon)}$.

\begin{theorem}\label{2065}
Let $\varphi$ be given by $(\ref{3455})$ and satisfying $\varphi(0)=0$. Let $(\varepsilon_n)_{n\ge0}$ be a 
decreasing sequence of real numbers such that $\lim_n\varepsilon_n=0$. 

Then on some probability space we can define an increasing sequence of L\'evy processes 
$(X^{(\varepsilon_n)})_{n\ge0}$, each process $X^{(\varepsilon_n)}$ having $\varphi^{(\varepsilon_n)}$ as 
Laplace exponent, and such that
$(S^{(\varepsilon_n)})_{n\ge0}:=(X^{(\varepsilon_{n+1})}-X^{(\varepsilon_{n})})_{n\ge0}$ is a sequence 
of independent subordinators that is itself independent of $X^{(\varepsilon_{0})}$. For each $n$, 
$S^{(\varepsilon_n)}$ has Laplace exponent,
\begin{equation}\label{3735}
\Gamma^{(\varepsilon_n)}(\lambda)=(\varepsilon_{n+1}-\varepsilon_n)\sigma^2\lambda+
\int_{(0,\infty)}\left(e^{-\lambda x}-1\right)\,\left(e^{-\varepsilon_{n+1} x}-
e^{-\varepsilon_n x}\right)\pi(dx).
\end{equation}
Moreover the sequence $(X^{(\varepsilon_{n})})$ converges uniformly over any closed intervals of 
$\mathbb{R}_+$, almost surely, toward a L\'evy process $X$ whose Laplace exponent is $\varphi$.
\end{theorem}

Let us now recall the Lamperti representation of a CSBP with branching mechanism $\varphi$ from the path
of a spectrally positive L\'evy process $X$ with Laplace exponent $\varphi$. This is the mapping 
defined as follows:
\begin{equation}\label{2675}
L(X)_t:=\left\{\begin{array}{ll}
X({I_t}\wedge \tau),\;\;&\mbox{if ${I_t}\wedge \tau<\infty$,}\\
\infty,\;\;&\mbox{if ${I_t}\wedge \tau=\infty$,}
\end{array}
\right.
\end{equation}
where $\displaystyle I_t=\inf\left\{s:\int_0^s\frac{du}{X_u}>t\right\}$ and 
$\tau=\inf\{t\ge0:X_t\le0\}$. Note that in (\ref{2675}), for $t>0$ and $X_0>0$, both variables $I_t$ and 
$\tau$ can be infinite only if $X$ drifts to $\infty$.
Then under $\p_x$, $x\ge0$, the process $(L(X)_t,\,t\ge0)$ is a CSBP with branching mechanism $\varphi$,
issued from $x$, see for instance Theorem 12.2, p.337 in \cite{ky}. Note also that the transformation 
(\ref{2675}) is invertible and the paths of the process $(X_t,\,0\le t\le\tau)$ can be recovered from 
those of the process $Z$.\\ 

Let us fix a Laplace exponent $\varphi$ as in $(\ref{3455})$ and a decreasing sequence 
$(\varepsilon_n)$ of real numbers such that $\lim_n\varepsilon_n=0$. Then on some probability space, 
we define L\'evy processes $X$ and $X^{(\varepsilon_n)}$, $n\ge0$ as in Proposition \ref{2065} and to 
make our notations simpler, we set,
\[\varphi^{(n)}:=\varphi^{(\varepsilon_{n})}\;\;\mbox{and}\;\;X^{(n)}:=X^{(\varepsilon_{n})}\,.\]
Furthermore from the paths of the sequence $(X^{(n)},\,n\ge0)$ and those of its limit $X$, we define the 
CSBP's $(Z^{(n)},\,n\ge0)$ and $Z$ through the Lamperti transformation:
\begin{equation}\label{6592}
Z^{(n)}:=L(X^{(n)})\,,\;n\ge0\;\;\;\mbox{and}\;\;\;Z:=L(X)\,.
\end{equation}
Given the convergence result obtained in Theorem \ref{2065}, it is natural to wonder if the sequence
$(Z^{(n)},\,n\ge0)$ converges toward $Z$ in some sense. This is the purpose of the next theorem which 
follows directly from Corollary 6 in \cite{cpu}. We shall denote by $\zeta$ the first hitting time of 
$\infty$ by the process $Z$, that is,
\begin{equation}\label{6532}
\zeta=\inf\{t\ge0:Z_t=\infty\}\,,
\end{equation}
where $\inf\emptyset=\infty$ (according to our notation $\{\zeta<\infty\}=A_\infty^\uparrow$). 
Note that under $\p_x$, the processes $Z^{(n)}$, $n\ge0$ and $Z$ are issued from $x$ as well as $X$ and 
$X^{(n)}$, $n\ge0$. Then from now on, we will only use the probabilities $\p_x$, $x\ge0$ 
for the processes $X$, $X^{(n)}$, $Z$ and $Z^{(n)}$. 

\begin{theorem}\label{3789}
For all $t,x>0$, the family of processes $(Z^{(n)}_s,0\le s\le t)$ converges $\p_x$-a.s.~on the set 
$\{t<\zeta\}$ in the Skohorod's $J_1$ topology toward the process $(Z_s,0\le s\le t)$.
\end{theorem}

\section{On the speed of explosion of the sequence $\{Z^{(n)},\,n\ge0\}$}\label{9936}

\subsection{Presentation of the problem}\label{9693}

Throughout the remainder of this article, we will assume that $\varphi$ is the Laplace exponent of a spLp 
satisfying the condition,
\begin{equation}\label{8378}
\mbox{$\varphi(0)=0$ and there is $\theta>0$, such that $\displaystyle \int_{0}^\theta
\frac{du}{|\varphi(u)|}<\infty$.}
\end{equation}
Then like in the previous section, we fix a decreasing sequence $(\varepsilon_n)$ such that 
$\lim_{n\rightarrow\infty}\varepsilon_n=0$ and we construct an increasing sequence of L\'evy 
processes $(X^{(n)})$ with respective Laplace exponents 
$\varphi^{(n)}(\lambda):=\varphi(\varepsilon_n+\lambda)-\varphi(\varepsilon_n)$, $\lambda\ge0$, as in 
Theorem \ref{2065}. Its limit, $X$, is then a L\'evy process with Laplace exponent $\varphi$. Note that
all the processes $X^{(n)}$, $n\ge0$ and $X$ drift to $\infty$. Recall also the construction (\ref{6592}), 
from the paths of $X^{(n)}$, $n\ge0$ and $X$, of the branching processes $Z^{(n)}$, $n\ge0$ and their 
limit $Z$ whose respective branching mechanism are $\varphi^{(n)}$, $n\ge0$ and $\varphi$.
We emphasize that, due to the time change in the transformation (\ref{6592}), the sequence $(Z^{(n)})$ 
is no longer monotone.\\

From Theorem \ref{5374} and our assumption (\ref{8378}), explosion occurs for the process $Z$, that is:
\[\mbox{for all $x>0$, $\p_x(\zeta<\infty)=\p_x(A_\infty^\uparrow)=\p_x(A_\infty)=1-e^{-x\rho}>0$},\]
where $\zeta$ is the explosion time defined in (\ref{6532}). Moreover, it can be seen from the Lamperti 
transformation (\ref{2675}) that $Z$ explodes in a 'continuous way', that is $Z_{\zeta-}=\infty$, 
$\p_x$-a.s. for all $x>0$ on the set $\{\zeta<\infty\}$, see Figure 1 below. Let us also point out on the
fact that, from Theorem \ref{2065}, the sequence of sets $\{\tau^{(n)}=\infty\}$ is increasing and tends 
to $\{\tau=\infty\}$, where $\tau^{(n)}=\inf\{t\ge0:X^{(n)}_t\le0\}$ and $\tau=\inf\{t\ge0:X_t\le0\}$. 
Moreover it appears from the construction (\ref{6592}) that 
\[\{\zeta<\infty\}=\{\tau=\infty\}\,.\]

On the other hand, since for all $n\ge0$, $|\varphi^{{(n)}'}(0)|=|\varphi'(\varepsilon_n)|<\infty$, Theorem 
\ref{5374} implies that explosion occurs for none of the processes $Z^{(n)}$, that is with obvious notation 
for all $n\ge0$ and $x\ge0$, $\p_x(A_\infty^{n,\uparrow})=0$. Note also that from (\ref{8378}), 
$\varphi'(0)=-\infty$, so that we can assume without loss of generality that 
$\varphi'(\lambda)<0$, for all $\lambda\in[0,\varepsilon_0]$. This implies that for all $n\ge0$, 
$\varphi^{{(n)}'}(0)=\varphi'(\varepsilon_n)<0$, so that 
$\rho_n:=\sup\{\lambda\ge0:\varphi^{(n)}(\lambda)\le0\}>0$. Therefore, each process $Z^{(n)}$ 
tends to $\infty$ with positive probability. This can be summarized as follows: 
\[\mbox{for all $x\ge0$ and $n\ge0$, $\p_x(A_\infty^{n,\uparrow})=0$ and 
$\p_x(A_\infty^{n})=1-e^{-x\rho_n}>0$}.\]

Let us now denote by $\sigma_{y}^{(n)}$ the first passage time by the process $Z^{(n)}$ above a level $y>0$, 
that is
\[\sigma_{y}^{(n)}=\inf\{t:Z^{(n)}_t\ge y\}\,,\]
and recall that from Theorem \ref{3789}, for all $t,x>0$, the sequence of non explosive processes 
$(Z^{(n)}_s,0\le s\le t)$, $n\ge0$, converges $\p_x$-a.s.~on the set $\{t<\zeta\}$ toward the explosive 
process $(Z_s,0\le s\le t)$ in the Skohorod's $J_1$ topology. In order to study the speed of explosion of the 
family $(Z^{(n)},\,n\ge0)$, we shall characterise the set of increasing sequences $(v_n)$ such that 
$\lim_{n\to \infty}v_n=\infty$ and for which the sequence of random variables $(\sigma_{v_n}^{(n)})$ converges 
in some sense. If $(v_n)$ is not too fast, then $(\sigma_{v_n}^{(n)})$ is expected to converge, whereas if 
the speed of convergence of $(v_n)$ is too fast compared to this of $(\varepsilon_n)$, then 
$\sigma_{v_n}^{(n)}\rightarrow\infty$ should go to $\infty$. Therefore, there must be a threshold for this 
speed under which $\sigma_{v_n}^{(n)}$ tends toward a non degenerate random variable which is actually 
expected to be the explosion time $\zeta$ of $Z$, see Figure 1 below.

\begin{figure}[H] 
\centering
\begin{tikzpicture}[scale=1]

\draw[->][line width=1] (-1,0) -- (12,0) node[right] {$t$};
\draw[->][line width=1] (0,-1) -- (0,8)  ;
\node[left,font=\small] at (0,1) {$x$} ;

\draw[dashed][line width=1] (4,-1) -- (4,8);
\node[right, red][scale=1.5] at (3.3,-0.4) {$\zeta$};
\node[right, red][scale=1.4] at (2.7,5) {$Z$};

\draw [red][line width=1.5] (0,1) -- (0.75,1.56);
\draw [red][line width=1.5] (0.75,1.56) -- (1.07,1.38);
\draw [red][line width=1.5] (1.07,1.38) -- (1.3,1.89);
\draw [red][line width=1.5] (1.33,2.33) -- (1.88,2.59);
\draw [red][line width=1.5] (1.88,2.59) -- (2.31,2.26);
\draw [red][line width=1.5] (2.31,2.26) -- (2.78,3.05);
\draw [red][line width=1.5] (2.78,3.05) -- (3.09,2.72);
\draw [red][line width=1.5] (3.09,2.72) -- (3.31,3.02);
\draw [red][line width=1.5] (3.31,3.37) -- (3.58,4.19);
\draw [red][line width=1.5] (3.58,4.19) -- (3.70,3.99);
\draw [red][line width=1.5] (3.70,3.99) -- (3.75,5.8)--(3.85,5.6);
\draw [red][line width=1.5] (3.85,5.9)--(3.94,7.8);

\draw[dashed][line width=1] (-0.5,7) -- (8.9,7) ;
\draw[dashed][line width=1] (8.7,-1) -- (8.7,7.2) ;
\node[left, font=\tiny, orange][scale=1.5]  at (8.8,-0.4) {$\sigma^{(n)}_{v_n}$};
\node[left, font=\tiny, orange][scale=1.5]  at (0.1,7.32) {$v_n$};
\node[right, orange][scale=1.2] at (9.5,6.7) {$Z^{(n)}$};

\draw [orange][line width=1.5] (0,1) -- (0.63,1.44);
\draw [orange][line width=1.5] (0.63,1.45) -- (0.86,1.30);
\draw [orange][line width=1.5] (0.86,1.30) -- (0.91,1.40);
\draw [orange][line width=1.5] (0.91,1.40) -- (1.10,1.20);
\draw [orange][line width=1.5] (1.10,1.20) -- (1.17,1.30);
\draw [orange][line width=1.5] (1.18,1.40) -- (1.36,1.80);
\draw [orange][line width=1.5] (1.36,1.80) -- (1.55,1.50);
\draw [orange][line width=1.5] (1.55,2.1) -- (1.8,2.8) -- (2.2,2) -- (2.5,2.35) -- (2.7,2.10);
\draw [orange][line width=1.5] (2.7,2.4) -- (3.2,2.7) -- (3.5,2.4) -- (3.8,2.8) -- (4.2,2.4) ;
\draw [orange][line width=1.5] (4.2, 2.8) -- (4.6 , 3.2 ) -- (5, 2.8) -- (5.6 , 2.9);
\draw [orange][line width=1.5] (5.6 , 3.4) -- (6.2, 4.2) -- (6.4, 4) -- (7, 5) -- (7.5, 4.8) -- (7.7,4.9) ;
\draw [orange][line width=1.5] (7.7, 5.4) -- (8.5,6.2) ;
\draw [orange][line width=1.5] (8.5, 6.7) -- (9,7.4) -- (9.4,6.8) -- (10, 7.4) ;

\draw[dashed][line width=1] (-0.5,5.9) -- (11,5.9);
\draw[dashed][line width=1] (10.8,-1) -- (10.8,6.15);
\node[right, font=\tiny, blue][scale=1.5]  at (10.7,-0.4) {$\sigma^{(m)}_{v_m}$};
\node[left, font=\tiny, blue][scale=1.5]  at (0.15,6.2) {$v_m$};
\node[right, blue][scale=1.2] at (11,5.7) {$Z^{(m)}$};

\draw [blue][line width=1.5] (0,1) -- (0.63,1.39);
\draw [blue][line width=1.5] (0.63,1.39) -- (0.86,1.07);
\draw [blue][line width=1.5] (0.86,1.07) -- (0.91,1.25);
\draw [blue][line width=1.5] (0.91,1.25) -- (1.10,1.02);
\draw [blue][line width=1.5] (1.10,1.02) -- (1.17,1.10);
\draw [blue][line width=1.5] (1.17,1.23) -- (1.36,1.42);
\draw [blue][line width=1.5] (1.36,1.42) -- (1.68,1.32);
\draw [blue][line width=1.5] (1.68,1.32) -- (1.80,1.53);
\draw [blue][line width=1.5] (1.80,1.53) -- (1.94,1.23) -- (2.5,2) -- (2.8, 1.8) -- (3,1.9) -- (3.3, 2.2);
\draw [blue][line width=1.5] (3.3, 2.4) -- (3.8, 2.9) -- (4.3, 2.7)-- (4.7, 3);
\draw [blue][line width=1.5] (4.7, 3.15) -- (5.2, 3.4)--(5.6,3) -- (6, 3.2);
\draw [blue][line width=1.5] (6,3.5) -- (6.2,3.7)-- (7,3.6) -- (7.8, 4.1)--(8.2,3.95) -- (8.5, 4.3);
\draw [blue][line width=1.5] (8.5, 4.55) -- (9,5.3)--(9.6,5.2) -- (9.8, 5.4) -- (10 , 5.2) ;
\draw [blue][line width=1.5] (10,5.4) -- (10.4, 5.8)-- (10.7,5.7) -- (11, 6.2)-- (11.2,6);

\end{tikzpicture}

\caption{$Z$ in red,  $Z^{(n)}$ in orange and $Z^{(m)}$ in blue for $n>m$.}

\end{figure}
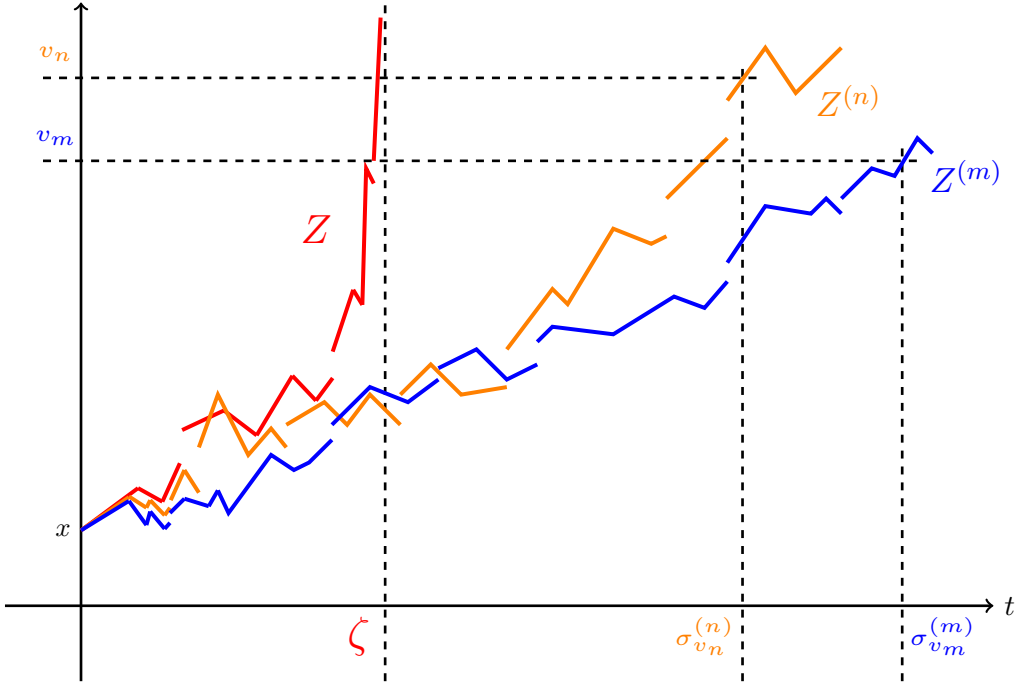

\subsection{Convergence in distribution}\label{8346}

In the remainder of this paper, we shall often use the notation,
\[\phi:=-\varphi\;\;\;\mbox{and}\;\;\;\phi^{(n)}:=-\varphi^{(n)}\,,\]
so that $\phi$ and $\phi^{(n)}$ are positive, concave functions on $(0,\rho)$ and $(0,\rho_n)$, 
respectively, where $\rho_n=\sup\{\lambda\ge0:\varphi^{(n)}\le0\}$. Moreover the sequence of functions 
$(\phi^{(n)})$ is increasing and satisfies $\lim_{n\rightarrow\infty}\phi^{(n)}=\phi$. We denote the 
Laplace exponent of each process $Z^{(n)}$ by $u^{(n)}$, see (\ref{3071}). In all the following 
statements, the function $h:\mathbb{Z}_+\rightarrow\mathbb{R}_+$ will always satisfy,
\begin{equation}\label{3930}
\mbox{$(h(n))_{n\ge0}$ is a decreasing sequence such that $\lim_{n\to\infty} h(n)=0$.}
\end{equation}
Let us also specify that when mentioning the existence of the limit of a non negative valued sequence, 
we mean that it exists in $[0,\infty]$.

\begin{theorem}\label{3332}
The following assertions are equivalent:
\begin{itemize}
\item[$1.$] There is a $[0,\infty]$-valued random variable $Y$ such that for all $k>0$ and for all $x\ge0$, 
under $\p_x$, $\sigma^{(n)}_{k/h(n)} \xrightarrow[n\to\infty]{(d)} Y$.
\item[$2.$] For all $\theta\in(0,\rho)$, 
$\displaystyle\lim_{n\to \infty}\int_{h(n)}^\theta \dfrac{du}{\phi^{(n)}(u)}$ exists.
\item[$3.$] For all $t\ge0$, $\lim_{n\to \infty}u_t^{(n)}(h(n)):=l_t(h)$ exists. 
\end{itemize}
When the above conditions are satisfied, the mapping $t\mapsto l_t(h)$ is continuous and non decreasing on 
$[0,\infty)$. It satisfies $l_0(h)=0$ and takes its values in $[0,\rho)$. Moreover, for all $x\ge0$, under 
$\p_x$,
\[Y\ed\zeta+c(h)\quad\mbox{and}\quad\lim_{n\to \infty}\int_{h(n)}^\theta \dfrac{du}{\phi^{(n)}(u)} = 
\int_{0}^\theta \dfrac{du}{\phi(u)}+c(h)\,,\]
where the constant $c(h)$ is given by $c(h)=\sup\{t\ge0:l_t(h)=0\}.$
\end{theorem}
\noindent Note that under our assumptions $\p_x(\zeta<\infty)>0$, hence the random variable $Y$ involved
in Theorem \ref{3332} satisfies $\p_x(Y=\infty)=1$ if and only if $c(h)=\infty$.
The following remark can be understood as a complement to Theorem \ref{3332}.

\begin{remark}
Given a constant $c\in[0,\infty]$ and $\theta\in(0,\rho)$, one can always construct a decreasing function 
$h_c$ such that
\begin{equation}
\lim_{n\to\infty}\int_{h_c(n)}^\theta \dfrac{du}{\phi^{(n)}(u)} = \int_0^\theta \dfrac{du}{\phi(u)}+ c\,.\label{8423}
\end{equation}
To this end, note that the mapping $\displaystyle x\mapsto F_n(x):= \int_x^\theta\dfrac{du}{\phi^{(n)}(u)}$ defines 
a decreasing bijection from $(0,\theta)$ to $(0,\infty)$. Let us denote by $F_n^{-1}$ its inverse. Then for 
$c\in[0,\infty)$, the sequence $h_c(n):=F^{-1}_n\left(\int_0^\theta \frac{du}{\phi(u)}+c\right)$ fulfils condition 
$(\ref{8423})$ and for $c=\infty$, we can take for instance 
$h_c(n):=F^{-1}_n\left(\int_0^\theta \frac{du}{\phi(u)}+n\right)$.
\end{remark}

\noindent Theorem \ref{3332} induces the following classification of sequences:
\begin{eqnarray*}
\mathcal{Z}&:=&\{h:\mbox{$(h(n))_{n\ge0}$ satisfies conditions $1.$, $2.$ and $3.$ of 
Theorem  \ref{3332}}\}\\
\mathcal{Z}_0&:=&\{h\in\mathcal{Z}:c(h)=0\}\,,\quad\mathcal{Z}_\infty:=\{h\in\mathcal{Z}:c(h)=\infty\}\\
\mathcal{Z}_c&:=&\{h\in\mathcal{Z}:0<c(h)<\infty\}\,.
\end{eqnarray*}

\noindent We shall prove in Lemma \ref{9331} below that a sufficient condition for $h\in\mathcal{Z}_0$ is,
\begin{equation}\label{2899}
\liminf_{n\rightarrow\infty}\phi(h(n)+\varepsilon_n)/\phi(\varepsilon_n)>1\,.
\end{equation}
Except in the regularly varying case, see Theorem \ref{2385} below, it is not clear that the weak 
convergence of $(\sigma^{(n)}_{k/h(n)})$ for some $k>0$ implies this convergence for all $k>0$, that is 
$h\in\mathcal{Z}_0$. However, if this convergence holds when $k$ is replaced by an independent exponentially 
distributed random variable, then $h\in\mathcal{Z}_0$ as shown in the following result. 

\begin{proposition}\label{0573}
Let ${\rm e}$ be an exponentially distributed random variable with parameter $1$ which is independent 
of the sequence $(Z^{(n)},\,n\ge1)$. Then 
$\sigma^{(n)}_{{\rm e}/h(n)} \xrightarrow[n\to\infty]{(d)} \zeta$, if and only if $h\in\mathcal{Z}_0$.
\end{proposition}

\noindent  The next results allow us to compare the limits of $(\sigma^{(n)}_{1/h(n)})$ and 
$(\sigma^{(n)}_{1/\tilde{h}(n)})$ according to the relative behaviours of the sequences $(h(n))$ and 
$(\tilde{h}(n))$.

\begin{proposition}\label{1442} Let $\tilde{h}:\mathbb{Z}_+\rightarrow\mathbb{R}_+$ 
be such that $(\tilde{h}(n))$ is decreasing with $\lim_{n\to\infty} \tilde{h}(n)=0$.
If $\sigma^{(n)}_{1/h(n)} \xrightarrow[n\to\infty]{(d)} \zeta$ and
$h(n)\le\tilde{h}(n)$, for all $n$ sufficiently large, then 
$\sigma^{(n)}_{1/\tilde{h}(n)} \xrightarrow[n\to\infty]{(d)} \zeta$.
\end{proposition}

\begin{proposition}\label{1332} Let $h\in\mathcal{Z}$ and $\tilde{h}:\mathbb{Z}_+\rightarrow\mathbb{R}_+$ 
be such that $(\tilde{h}(n))$ is a decreasing sequence with $\lim_{n\to\infty} \tilde{h}(n)=0$. 
\begin{itemize}
\item[$1.$] If $\tilde{h}\asymp h$, then $\tilde{h}\in\mathcal{Z}$ and $c(\tilde{h})=c(h)$. 
\item[$2.$] If $h\in\mathcal{Z}_c$ and if $\tilde{h}\in\mathcal{Z}$ satisfies 
$h(n)\le\tilde{h}(n)$ for all $n$ sufficiently large, then $c(\tilde{h})\le c(h)$.
\item[$3.$] If $h\in\mathcal{Z}_0$ and if $\tilde{h}$ satisfies 
$h(n)\le\tilde{h}(n)$ for all $n$ sufficiently large, then $\tilde{h}\in\mathcal{Z}_0$.
\item[$4.$] If $h\in\mathcal{Z}_\infty$ and if $\tilde{h}$ satisfies 
$\tilde{h}(n)\le h(n)$ for all $n$ sufficiently large, then $\tilde{h}\in\mathcal{Z}_\infty$.
\end{itemize}
\end{proposition}

\noindent  We can see in assertion 2.~of Proposition \ref{1332} that if $h\in\mathcal{Z}_c$, 
then it is not enough that $h\le\tilde{h}$ or $\tilde{h}\le h$ to deduce something on $\tilde{h}$, contrary 
to what is asserted in parts $3.$ and $4.$ Actually the proof of assertion 2.~requires the additional 
hypothesis that $\tilde{h}\in\mathcal{Z}$.\\ 

The following proposition suggests that the existence of the limit of 
$|\log{h(n)}|/\phi'({\varepsilon_n})$ as $n$ tends to $\infty$ might be a better criterion for 
the weak convergence of $(\sigma^{(n)}_{k/h(n)})$ than the convergence of 
$\displaystyle\int_{h(n)}^\theta \dfrac{du}{\phi^{(n)}(u)}$,
as stated in Theorem \ref{3332}. This can actually be proved in the regularly varying case, see 
Theorem \ref{2385}. 

\begin{proposition}\label{1100} $\mbox{}$ 
\begin{itemize}
\item[$1.$] If $\displaystyle\lim_{n\to\infty}\dfrac{|\log{h(n)}|}{\phi'({\varepsilon_n})}=\infty$, 
then $h\in\mathcal{Z}_\infty$.
\item[$2.$] If $h\in\mathcal{Z}_0$, then 
$\displaystyle\lim_{n\to\infty}\dfrac{|\log{h(n)}|}{\phi'({\varepsilon_n})}=0$.
\end{itemize}
\end{proposition}

In the case where $\varphi$ is regularly varying at $0$, a better criterion allows us to characterize 
the limit of $\sigma_{1/h(n)}^{(n)}$.

\begin{theorem}\label{2385}
Suppose that $\varphi$ is regularly varying at $0$. Then the following assertions are equivalent:
\begin{itemize}
\item[$1.$] The limit $c(h):=\displaystyle\lim_{n\to\infty} \dfrac{|\log{h(n)}|}{\phi'({\varepsilon_n})}$ 
exists. 
\item[$2.$] For all $k>0$, $\sigma^{(n)}_{k/h(n)} \xrightarrow[n\to\infty]{(d)} \zeta + c(h)$, that is 
$h\in\mathcal{Z}$.
\end{itemize}
When these assumptions are satisfied, the constant the constant $c(h)$ is the same as that in
Theorem $\ref{3332}$.
Moreover assertions $1.$~and $2.$~for $c(h)=0$ are equivalent to the following one: 
\begin{itemize}
\item[$3.$] For some $k>0$, $\sigma^{(n)}_{k/h(n)} \xrightarrow[n\to\infty]{(d)} \zeta$.
\end{itemize}
\end{theorem}

\subsection{Convergence in $L_1$ and almost sure convergence}\label{8347}

We recall the usual convention $0\times\infty=\infty\times 0=0$, so that in particular, if $Y$ is any  
non negative random variable, then $Y\ind_{\{Y<\infty\}}$ is a.s.~finite. 
As we will see in this section, when $h\in\mathcal{Z}_0$, stronger convergences than weak convergence 
actually hold. Let us first notice that the time $\zeta$ admits all its moments on the set 
$\{\zeta<\infty\}$. This result is also proved in Theorem 3.1 of \cite{lz} where the moments of $\zeta$ 
are expressed in a different form, see also Theorems 1.5 and 1.6 in \cite{li} where this result is 
obtained in a particular case.
 
\begin{proposition}\label{8453}
For all $n\in\mathbb{Z}_+$ and $x>0$, $\e_x\left(\zeta^n\ind_{\{\zeta<\infty\}}\right)<\infty$. 
Moreover for all $n\in\mathbb{Z}_+$, $x>0$ and $\lambda\ge0$, 
\begin{equation*}
\e_x\left(\zeta^n\ind_{\{\zeta<\infty\}}\right) = x\int_0^\rho F(y)^n e^{-xy} dy\;\;\;\mbox{and}\;\;\;
\e_x\left(e^{-\lambda\zeta}\right) = x\int_0^\rho e^{-\lambda F(y)}e^{-xy}dy\,,
\end{equation*}
where $F(y) = \int_0^y \frac{du}{\phi(u)}$, for $y \in [0,\rho)$. In particular, 
$\displaystyle\e_x\left(\zeta\ind_{\{\zeta<\infty\}}\right)=\int_0^{\rho}
\frac{e^{-\lambda x}-e^{-\rho x}}{\phi(\lambda)}d\lambda$.
\end{proposition}

\noindent As we will show in Subsection \ref{5935}, the first passage times $\sigma_{1/h(n)}^{(n)}$ 
are integrable on the set $\{\sigma_{1/h(n)}^{(n)}<\infty\}$ and we will prove the following result. 

\begin{theorem}\label{0486}
Assume that $h\in\mathcal{Z}_0$. Then for all $x>0$ and $n\ge0$, the random variable 
$\sigma_{1/h(n)}^{(n)}\mathds{1}_{\{\sigma_{1/h(n)}^{(n)}<\infty\}}$ is integrable under $\p_x$, moreover, 
\[\sigma_{1/h(n)}^{(n)}\mathds{1}_{\{\sigma_{1/h(n)}^{(n)}<\infty\}}\xrightarrow[n\to\infty]
{\mathrm{L}^1(\p_x)}\zeta\mathds{1}_{\{\zeta<\infty\}}\,.\]
\end{theorem}

\noindent It follows from Theorems \ref{3332} and \ref{0486} that for $h\in\mathcal{Z}_0$, the
convergence in law of $\left(\sigma_{1/h(n)}^{(n)}\right)$ toward $\zeta$ is equivalent to the convergence in 
$L^1$ of $\left(\sigma_{1/h(n)}^{(n)}\mathds{1}_{\{\sigma_{1/h(n)}^{(n)}<\infty\}}\right)$ toward 
$\zeta\mathds{1}_{\{\zeta<\infty\}}$.\\

We also emphasize that if $h\in\mathcal{Z}_c$, that is if 
$\sigma^{(n)}_{1/h(n)} \xrightarrow[n\to\infty]{(d)} \zeta + c(h)$, with $c(h)\in(0,\infty)$, then 
the convergence in $L^1$ of $\left(\sigma_{1/h(n)}^{(n)}\mathds{1}_{\{\sigma_{1/h(n)}^{(n)}<\infty\}}\right)$ 
toward $\zeta\mathds{1}_{\{\zeta<\infty\}}+c(h)$ would make sense only if the constant $c(h)$ can be obtained 
as a functional of the paths of the sequence $(Z^{(n)})$ and those of the process $Z$. This does not seem very
realistic and the constant $c(h)$ may just be an adjustment value that appears in the weak limit, when $h$ 
belongs to the critical domain~$\mathcal{Z}_c$. The same remark obviously holds for almost sure convergence, 
below.\\

We now state that under a stronger assumption than this of Theorem \ref{0486}, the sequence 
$\left(\sigma_{1/h(n)}^{(n)}\ind_{\{\sigma_{1/h(n)}^{(n)}<\infty\}}\right)$ converges almost surely 
toward $\zeta\ind_{\{\zeta<\infty\}}$. 

\begin{theorem}\label{1086}
Assume that $\sum \phi(\varepsilon_n)/\phi(\varepsilon_n+h(n))<\infty$. Then for all $x>0$, 
\begin{eqnarray*}
\sigma_{1/h(n)}^{(n)}\ind_{\{\sigma_{1/h(n)}^{(n)}<\infty\}}\xrightarrow[n\to\infty]{\p_x-a.s.}
\zeta\ind_{\{\zeta<\infty\}}\,.
\end{eqnarray*}
\end{theorem}
The condition $\sum \phi(\varepsilon_n)/\phi(\varepsilon_n+h(n))<\infty$ is not a necessary condition for
the almost sure convergence of $\left(\sigma_{1/h(n)}^{(n)}\right)$ to hold. Indeed,
take for instance $\phi(\lambda)=k\lambda^{1/2}$, $\lambda\ge0$, where $k$ is some positive constant.
Take also $h(n)=\varepsilon_n=n^{-1}$. Then from Theorems \ref{2385} and \ref{0486}, there is a sequence
of integers $(n_i)_{i\ge0}$ with $n_i\rightarrow\infty$, as $i\rightarrow\infty$ such that 
$\sigma_{1/h(n_i)}^{(n_i)}$ converges almost surely toward $\zeta$, as $i\rightarrow\infty$. However 
$\sum_i\phi(\varepsilon_{n_i})/\phi(\varepsilon_{n_i}+h(n_i))=\infty$. This provides a counterexample 
with $\phi(\lambda)=k\lambda^{1/2}$, $\tilde{\varepsilon}_i:=\varepsilon_{n_i}$ and $\tilde{h}(i):=h(n_i)$.\\

From Theorem \ref{1086}, convergence in distribution of $(\sigma_{1/h(n)}^{(n)})$ holds, that is 
$h\in\mathcal{Z}_0$, under the condition $\sum \phi(\varepsilon_n)/\phi(\varepsilon_n+h(n))<\infty$. 
This can be checked directly by applying Lemma \ref{9331}, see (\ref{2899}) above.\\

We end this section with a related result that provides another way to evaluate the 
speed of convergence of the sequence $(Z^{(n)})$ toward $Z$. Let us give ourselves an 
exponentially distributed random variable ${\rm e}$ with parameter 1 
that is independent of the sequence of L\'evy processes $(X^{(n)})$. For each $n\ge0$, denote by 
$\tilde{X}^{(n)}$ the L\'evy process ${X}^{(n)}$ killed at the independent exponential time 
${\rm e}_n:={\rm e}/h(n)$. The killed L\'evy process $\tilde{X}^{(n)}$ has Laplace exponent 
$\tilde{\varphi}^{(n)}(\lambda):=\varphi^{(n)}(\lambda)-h(n)$. Then we define the sequence 
$\tilde{Z}^{(n)}:=L(\tilde{X}^{(n)})$, $n\ge0$ of branching processes obtained from the sequence 
of L\'evy processes $\tilde{X}^{(n)}$ through the Lamperti representation (\ref{2675}). The process
$\tilde{Z}^{(n)}$ has branching mechanism $\tilde{\varphi}^{(n)}$. Recall from 
Theorem \ref{5374} that $\tilde{Z}^{(n)}$ hits $\infty$ in a finite time with positive probability 
and from (\ref{2675}) this is done through a jump, that is 
$\tilde{Z}^{(n)}_{\tilde{\zeta}^{(n)}-}<\infty$, $\p_x$-a.s.~on the set 
$\{\tilde{\zeta}^{(n)}<\infty\}$, for all $x>0$, where 
$\tilde{\zeta}^{(n)}:=\inf\left\{t:\tilde{Z}^{(n)}_t=\infty\right\}$. Actually, the process 
$\tilde{Z}^{(n)}$ corresponds to the process ${Z}^{(n)}$ killed at the time 
$\int_0^{{\rm e}_n\wedge \tau^{(n)}}\frac{du}{X^{(n)}_u}$, where $\tau^{(n)}:=
\inf\left\{t: X^{(n)}_t\le0\right\}$. Moreover the hitting time $\tilde{\zeta}^{(n)}$ satisfies
\begin{equation}\label{5440}
\tilde{\zeta}^{(n)}\mathds{1}_{\{\tilde{\zeta}^{(n)}<\infty\}}=
\int_0^{{\rm e}_n}\frac{du}{X^{(n)}_u}\ind_{\{{\rm e}_n<\tau^{(n)}\}}.
\end{equation}
By reinforcing the hypothesis of Theorem \ref{1086}, we can prove that the sequence of explosion 
times $(\tilde{\zeta}^{(n)}\ind_{\{\tilde{\zeta}^{(n)}<\infty\}})$ converges almost surely 
toward $\zeta\ind_{\{\zeta<\infty\}}$.

\begin{proposition}\label{0456}
If $\sum \phi(\varepsilon_n)/h(n)<\infty$, then for all $x>0$,
\begin{eqnarray*}
\tilde{\zeta}^{(n)}\ind_{\{\tilde{\zeta}^{(n)}<\infty\}}
\xrightarrow[n\to\infty]{\p_x-a.s.}\zeta\ind_{\{\zeta<\infty\}}\,.
\end{eqnarray*}
\end{proposition}

\noindent Proposition \ref{0456} raises the question of the convergence in $L^1$ of the sequence 
$(\tilde{\zeta}^{(n)}\ind_{\{\tilde{\zeta}^{(n)}<\infty\}})$. When $X$ is a subordinator and under 
the sole assumption that $h\in\mathcal{Z}_0$, this is a direct consequence of Corollary \ref{1153}. 
The general case seems much more delicate. 

\section{Proofs}\label{1109}

\subsection{Proof of Theorem $\ref{2065}$}

Let us first write $\varphi^{(\varepsilon_n)}$ as follows, 
\begin{eqnarray*}
\varphi^{(\varepsilon_n)}(\lambda)&=&\varphi(\lambda+\varepsilon_n)-\varphi(\varepsilon_n)\\
&=&\left(a+\varepsilon_n\sigma^2+\int_{(0,\infty)}x\left(1-e^{-\varepsilon_n x}\right)\ind_{\{x<1\}}\,
\pi(dx)\right)\lambda+\frac12\sigma^2\lambda^2\\
&&+\int_{(0,\infty)}\left(e^{-\lambda x}-1+\lambda x\ind_{\{x<1\}}\right)\,e^{-\varepsilon_n x}\pi(dx)\,.
\end{eqnarray*}
This yields that the difference  
\begin{eqnarray*}
\Gamma^{(\varepsilon_n)}(\lambda)&:=&\varphi^{(\varepsilon_{n+1})}(\lambda)-
\varphi^{(\varepsilon_n)}(\lambda)\\
&=&(\varepsilon_{n+1}-\varepsilon_n)\sigma^2\lambda+\int_{(0,\infty)}\left(e^{-\lambda x}-1\right)\,
\left(e^{-\varepsilon_{n+1} x}-e^{-\varepsilon_n x}\right)\pi(dx) 
\end{eqnarray*}
is the Laplace exponent of a subordinator.\\ 

Then on the same probability space, define a L\'evy process $X^{(\varepsilon_0)}$ with Laplace 
exponent $\varphi^{(\varepsilon_0)}$ and a sequence of independent subordinators $(S^{(\varepsilon_n)})_{n\ge0}$ 
such that for each $n$, $S^{(\varepsilon_n)}$ has Laplace exponent $\Gamma^{(\varepsilon_n)}$. Assume moreover 
that $X^{(\varepsilon_0)}$ is independent of the sequence $(S^{(\varepsilon_n)})_{n\ge0}$. Then the 
sequence of L\'evy processes $X^{(\varepsilon_n)}:=X^{(\varepsilon_0)}+\sum_{k=0}^{n-1} S^{(\varepsilon_k)}$
satisfies the conditions of the statement.\\ 

Note that $\Sigma^{(n)}:=\sum_{k=0}^{n-1} S^{(\varepsilon_k)}$ is an increasing sequence of subordinators 
and that $\Sigma^{(n)}$ has Laplace exponent, $\varphi^{(\varepsilon_{n})}-\varphi^{(\varepsilon_0)}$. 
Moreover, $\lim_{n\rightarrow\infty}\varphi^{(\varepsilon_{n})}=\varphi$. 
Therefore, for each $t\ge0$, $(\Sigma^{(n)}_t)$ converges weakly toward an infinitely divisible distribution 
having $\varphi-\varphi^{(\varepsilon_0)}$ as Laplace exponent. Then for each $t\ge0$, the a.s.~limit 
$\Sigma_t=\sum_{k=0}^{\infty} S^{(\varepsilon_k)}_t$
of $\Sigma^{(n)}_t$  is a.s.~finite and has Laplace exponent $\varphi-\varphi^{(\varepsilon_0)}$. 
Moreover $(\Sigma_t)$ defines a c\`adl\`ag non decreasing process. It is actually a subordinator with Laplace 
exponent $\varphi-\varphi^{(\varepsilon_0)}$. Since for all $t\ge0$, 
$\sup_{0\le s\le t}|\Sigma_s^{(n)}-\Sigma_s|=|\Sigma_t^{(n)}-\Sigma_t|$, the sequence of subordinators
$(\Sigma^{(n)})$ converges uniformly over any closed interval of $\mathbb{R}_+$, almost surely, toward 
the process $\Sigma$. 
This implies the last assertion of Theorem \ref{2065} regarding the sequence $(X^{(\varepsilon_{n})})$. 
$\hfill\Box$

\subsection{Proof of the results in Subsection $\ref{8346}$ (convergence in distribution)}
Recall that $\phi^{(n)}=-\varphi^{(n)}$, $n\ge0$ and $\phi=-\varphi$ are the branching mechanisms 
of $Z^{(n)}$, $n\ge0$ and $Z$, respectively and that 
$(u^{(n)}_t(\lambda),\,t,\lambda\ge0)$, $n\ge0$ and $(u_t(\lambda),\,t,\lambda\ge0)$
are the corresponding Laplace exponents. Then we first state some of the elementary properties
of these Laplace exponents, a part of which can also be found in Chapter 3 of \cite{li}.

\begin{lemma}\label{9522} $\mbox{}$
\begin{itemize}
\item[$1.$] For all $t\ge0$ and $n\ge0$, the mappings $\lambda\mapsto u_t^{(n)}(\lambda)$ and 
$\lambda\mapsto u_t(\lambda)$ are continuous and increasing on $[0,\infty)$.
\item[$2.$] For all $\lambda\in(0,\rho)$, the mapping $t\mapsto u_t(\lambda)$ is continuous and increasing 
on $[0,\infty)$. Moreover it is valued in the set $(0,\rho)$. 
\item[$3.$] For all $\lambda\in(0,\rho)$, there is $n_\lambda\ge1$ such that for all $n\ge n_\lambda$, 
the mappings $t\mapsto u_t^{(n)}(\lambda)$ are continuous and increasing on $[0,\infty)$.
Moreover they are valued in the set $(0,\rho)$.
\item[$4.$] Let $\rho_0$ be the largest root of $\phi^{(0)}$. Then for all $t\ge0$ and 
$\lambda\in(0,\rho_0)$, the sequence $(u_t^{(n)}(\lambda))_{n\ge0}$ is non decreasing and 
$\lim_{n\rightarrow\infty}u_t^{(n)}(\lambda)=u_t(\lambda)$.
\end{itemize}
\end{lemma}
\begin{proof}
The first assertion follows directly from the definition (\ref{3071}) as well as continuity of the mappings
in $2.$ and $3.$  

We first prove that for $\lambda\in(0,\rho)$, the mapping $t\mapsto u_t(\lambda)$ is increasing on $[0,\infty)$. 
From (\ref{3451}), $\frac{\partial u_t}{\partial t}(\lambda)|_{t=0}=\phi(u_0(\lambda))=\phi(\lambda)>0$, 
so that $u_\delta(\lambda)>u_0(\lambda)=\lambda$, for all $\delta>0$ small enough. 
Let $t>0$, then from the mean value theorem applied to the function $\lambda\mapsto u_t(\lambda)$, 
there is $\lambda_\delta\in[\lambda,u_\delta(\lambda)]$, such that,
\[u_t(u_\delta(\lambda))-u_t(\lambda)=u_{t+\delta}(\lambda)-u_t(\lambda)=
\frac{\partial}{\partial\lambda}u_t(\lambda_\delta)(u_\delta(\lambda)-\lambda),\]
where we have used the semigroup property of $u_t(\lambda)$ in the first equality. Note that both 
$u_\delta(\lambda)$ and $\lambda_\delta$ tend to $\lambda$, as $\delta$ tends to 0, hence it follows 
from the above equality that, 
\[\lim_{\delta\rightarrow0}\frac{u_{t+\delta}(\lambda)-u_t(\lambda)}{u_\delta(\lambda)-\lambda}=
\frac{\partial}{\partial\lambda}u_t(\lambda)>0\,.\]
Multiplying and dividing the left hand side of this equality by $\delta$ yields, 
\[\lim_{\delta\rightarrow0}\frac{u_{t+\delta}(\lambda)-u_t(\lambda)}{u_\delta(\lambda)-\lambda}=
\frac{\partial}{\partial t}u_t(\lambda)\left(\frac{\partial u_t}{\partial t}(\lambda)|_{t=0}\right)^{-1}>0\,.\]
Hence, $\frac{\partial}{\partial t}u_t(\lambda)$ has the same sign as 
$\frac{\partial u_t}{\partial t}(\lambda)|_{t=0}=\phi(\lambda)>0$. This proves that $t\mapsto u_t(\lambda)$ 
is increasing on $[0,\infty)$. Then it is clear that of all $\lambda\in(0,\rho)$, $u_t(\lambda)>0$.
Moreover $u_t(\lambda)<\rho$ since $\frac{\partial u_t}{\partial t}(\lambda)=\phi(u_t(\lambda))>0$.

Let $\rho_n$ be the largest root of $\phi^{(n)}$, then since $(\phi^{(n)})$ is increasing and 
$\phi^{(n)}\le\phi$, the sequence $(\rho_n)$ is increasing and $\rho_n\le \rho$. Moreover, 
$\lim_{n\rightarrow\infty}\rho_n=\rho$, so that for all $\lambda\in(0,\rho)$, there is $n_\lambda\ge1$ 
such that for all $n\ge n_\lambda$, $\lambda\in(0,\rho_n)$. Then the proof of 3.~follows this of 2.~along 
the lines.

Let us emphasize that from $2.$ and $3.$, for $n\ge n_\lambda$, $\phi^{(n)}$ has no root between 
$\lambda\in(0,\rho_n)$ and $u_t^{(n)}(\lambda)$ and $\phi$ has no root between $\lambda\in(0,\rho)$ and 
$u_t(\lambda)$. Since $\phi^{(n)}\le\phi^{(n+1)}$ for all $n$, we derive from (\ref{3453}) that for all 
$n\ge0$, $\lambda\in(0,\rho_n)$ and $t\ge0$, 
\[\int_{\lambda}^{u_t^{(n+1)}(\lambda)}\dfrac{du}{\phi^{(n+1)}(u)}=t=
\int_{\lambda}^{u_t^{(n)}(\lambda)}\dfrac{du}{\phi^{(n)}(u)}\le 
\int_{\lambda}^{u_t^{(n+1)}(\lambda)}\dfrac{du}{\phi^{(n)}(u)}\,.\]
Therefore, $u_t^{(n)}(\lambda) \le u_t^{(n+1)}(\lambda)$, for all $\lambda\in(0,\rho_n)$. On the other hand,
we derive from 
\begin{equation}\label{2067}
\int_{\lambda}^{u_t^{(n)}(\lambda)}\dfrac{du}{\phi(u)}\le\int_{\lambda}^{u_t^{(n)}(\lambda)}\dfrac{du}{\phi^{(n)}(u)}
=t=\int_{\lambda}^{u_t^{(n)}(\lambda)}\dfrac{du}{\phi(u)}+\int_{u_t^{(n)}(\lambda)}^{u_t(\lambda)}\dfrac{du}{\phi(u)}
\end{equation}
that $u^{(n)}_t(\lambda)\le u_t(\lambda)$, for all $n\ge0$ and $\lambda\in(0,\rho_n)$. Moreover, 
since $u_t(\lambda)<\rho$, we derive from (\ref{2067}) and dominated convergence theorem that for all 
$\lambda\in(0,\rho_0)$, $\lim_{n\rightarrow\infty}u_t^{(n)}(\lambda)=u_t(\lambda)$.
\end{proof}

Recall that $(h(n))_{n\ge0}$ is a decreasing sequence such that $\lim_{n\to \infty} h(n)=0$. Moreover, with no
loss of generality, we shall henceforth assume that $h(0)<\rho_0$.\

\begin{lemma}\label{3420}$\mbox{}$
\begin{itemize}
\item[$1.$] For all $t\ge0$ and $n\ge0$, $0<u^{(n)}_t(h(n))\le u_t(h(0))<\rho$. 
\item[$2.$] Assume that $(u^{(n)}_t(h(n)))_{n\ge0}$ converges for all $t\ge0$ and set
\[l_t(h):=\lim_{n\rightarrow\infty}u_t^{(n)}(h(n))\,.\]
Then the mapping $t\mapsto l_t(h)$ is continuous and non decreasing on $[0,\infty)$. In this case, we set 
\[c(h):=\sup\{t\ge0:l_t(h)=0\}\,.\]
\item[$3.$]
The sequence $(u^{(n)}_t(h(n)))_{n\ge0}$ converges for all $t\ge0$ if and only if 
$\displaystyle\lim_{n\to\infty}\int_{h(n)}^\theta\dfrac{du}{\phi^{(n)}(u)}$ exists in $[0,\infty]$ for some 
$\theta\in(0,\rho)$. In this case, 
\begin{equation}\label{4005}
\lim_{n\to\infty}\int_{h(n)}^\theta\dfrac{du}{\phi^{(n)}(u)}=c(h)+\int_0^\theta\frac{du}{\phi(u)},\quad 
\mbox{for all $\theta\in(0,\rho)$.}
\end{equation}
If $c(h)<\infty$, then for all $t\ge c(h)$,
\begin{equation*}
\lim_{n\to \infty} u_t^{(n)}(h(n))=u_{t-c(h)}(0).
\end{equation*}
\end{itemize}
\end{lemma}
\begin{proof} Since $(h(n))$ is decreasing and $h(0)\in(0,\rho_0)$, from $1.$ and $4.$ in Lemma \ref{9522}, 
for all $t\ge0$, $n\ge0$ and $k\ge n$, $u^{(n)}_t(h(n))\le u_t^{(n)}(h(0))\le u_t^{(k)}(h(0))$. Then by 
letting $k$ go to $\infty$ and applying $2.$ in Lemma \ref{9522} again, we obtain 
$u^{(n)}_t(h(n))\le u_t^{(n)}(h(0))\le u_t(h(0))$.
Since $h(0)\in(0,\rho_0)$ and $\rho_0<\rho$, assertion $2.$ in Lemma \ref{9522} implies that 
$u_t(h(0))<\rho$.

Let us now prove assertion 2. Assume that $(u^{(n)}_t(h(n)))_{n\ge0}$ converges for all $t\ge0$. Then 
from $3.$ in Lemma \ref{9522}, for all $\lambda\in(0,\rho)$ and $n$ sufficiently large, the functions 
$t\mapsto u_t^{(n)}(\lambda)$ are increasing, therefore $t\mapsto l_t(h)$ is non decreasing. 
Moreover $t\mapsto l_t(h)$ is continuous. Indeed from (\ref{3453}) applied to $\phi^{(n)}$, for all 
$n\ge0$ and $0\le s\le t$,
\[\int_{u_s^{(n)}(h(n))}^{u_t^{(n)}(h(n))}\dfrac{du}{\phi^{(n)}(u)}=t-s\,,\]
so that from Fatou's lemma and the assumption 
\[\int_{l_s(h)}^{l_t(h)}\dfrac{du}{\phi(u)}\le
\liminf_{n}\int_{u_s^{(n)}(h(n))}^{u_t^{(n)}(h(n))}\dfrac{du}{\phi^{(n)}(u)}=t-s\,.\]
The continuity of $t\mapsto l_t(h)$ follows from this.

Let us finally prove 3. Assume first that $l_t(h)=\lim_{n\rightarrow\infty}u^{(n)}_t(h(n))$ exists for all 
$t\ge0$. If $c(h)=\infty$, that is for all $t\ge0$, $l_t(h)=0$, then by Fatou's Lemma in 
\begin{equation}\label{6749}
     \int_{h(n)}^\theta \dfrac{du}{\phi^{(n)}(u)}= t + \int^\theta_{u^{(n)}_t(h(n))} 
     \dfrac{du}{\phi^{(n)}(u)}\,,\;\;\;\theta\in(0,\rho_n)\,,
\end{equation}
we obtain, for all $\theta\in(0,\rho)$, 
\begin{equation*}
\liminf_{n\to\infty}\int_{h(n)}^\theta\dfrac{du}{\phi^{(n)}(u)}\geq t+\int_0^\theta\dfrac{du}{\phi(u)}
\end{equation*}
and taking the limit when $t$ goes to infinity implies $\displaystyle\lim_{n\to\infty}  
\int_{h(n)}^\theta \dfrac{du}{\phi^{(n)}(u)}=\infty$. If $c(h)<\infty$, then let $t>c(h)$, so that 
$l_t(h)>0$. Using dominated convergence in ($\ref{6749}$)
yields for all $\theta\in(0,\rho)$, 
$\displaystyle\lim_{n\to\infty}\int_{h(n)}^\theta \dfrac{du}{\phi^{(n)}(u)} = t + 
\int_{l_t(h)}^\theta \dfrac{du}{\phi(u)}$, and letting $t$ tend to $c(h)$ in this equality, we obtain 
by continuity of the function $t\mapsto l_t(h)$ that 
$\displaystyle\lim_{n\to\infty}\int_{h(n)}^\theta \dfrac{du}{\phi^{(n)}(u)}=
c(h)+\int_0^\theta\frac{du}{\phi(u)}$. This implies that,
\begin{equation*}
    \int_0^{l_t(h)} \frac{du}{\phi(u)}= t-c(h),
\end{equation*}
for all $t\ge c(h)$, that is $l_t(h)=u_{t-c(h)}(0)$. 

Conversely, assume that 
$\displaystyle\lim_{n\to\infty}\int_{h(n)}^\theta\dfrac{du}{\phi^{(n)}(u)}$ exists in $[0,\infty]$ and 
that for some $t>0$, there are two subsequences $(v_t^{(n)})$ and 
$(w_t^{(n)})$ of $(u_t^{(n)}(h(n)))$ which converge toward $v_t>0$ and 0 respectively. Then from 
dominated convergence in (\ref{6749}) where we replaced $u_t^{(n)}(h(n))$ by $v_t^{(n)}$ we obtain,
\begin{equation}\label{0200}
\lim_{n\to\infty}\int_{h(n)}^\theta\dfrac{du}{\phi^{(n)}(u)}=t+\int_{v_t}^\theta\dfrac{du}{\phi(u)},
\end{equation}
and from Fatou's lemma in (\ref{6749}) where we replaced 
$u_t^{(n)}(h(n))$ by $w_t^{(n)}$ we obtain,
\[\lim_{n\to\infty}\int_{h(n)}^\theta\dfrac{du}{\phi^{(n)}(u)}\geq t+\int_0^\theta\dfrac{du}{\phi(u)}.\]
This is contradictory, so either any subsequence of $(u_t^{(n)}(h(n)))$ converges toward a positive 
limit, and in this case, from (\ref{0200}) these limits are identical, or  $(u_t^{(n)}(h(n)))$ converges 
toward~0. If $t=0$, then $u_t^{(n)}(h(n))=h(n)\rightarrow 0=u_0(0)$. Finally, (\ref{4005}) follows by 
taking $t=c(h)$ in (\ref{0200}), where actually $v_t=l_t(h)$ and by recalling that from part 2., 
$l_{c(h)}(h)=0$. 
\end{proof}

\begin{lemma}\label{CVPS} For all $t\ge0$ and $x\ge0$,
\[\lim_{n\rightarrow\infty}Z^{(n)}_t=Z_t\;\;\;\mbox{and}\;\;\;\lim_{n\rightarrow\infty}
\overline{Z}^{(n)}_t=\overline{Z}_t\,,\;\;\;\mbox{$\p_x$-a.s.}\]
Moreover, on the event $\{\zeta\leq t\}$,
\begin{align*}
\lim_{n\to \infty } \dfrac{Z^{(n)}_t}{\overline{Z}^{(n)}_t} = 1\,,\;\;\;\mbox{$\p_x$-a.s.}
\end{align*}
\end{lemma}
\begin{proof} Fix $t\ge0$, $x\ge0$ and let $(\mathcal{F}_t)$ be the filtration generated by all the processes
$X^{(n)}$, $n\ge0$. Define also $\displaystyle I^{(n)}_t=\inf\left\{s:\int_0^s\frac{du}{X^{(n)}_u}>t\right\}$ 
and $\tau^{(n)}=\inf\{t:X^{(n)}_t\le0\}$.
Then for each $t\ge0$, the sequence $(I^{(n)}_t\wedge\tau^{(n)})$ is an increasing sequence of stopping times 
in the filtration $(\mathcal{F}_t)$, and it satisfies $\lim_n I^{(n)}_t\wedge\tau^{(n)}=I_t\wedge\tau$, 
$\p_x$-a.s. Moreover, $X$ is a L\'evy process in the filtration $(\mathcal{F}_t)$ and it is 
quasi-left-continuous. Therefore $\lim_n X(I^{(n)}_t\wedge\tau^{(n)})=X(I_t\wedge\tau)$, $\p_x$-a.s.~on the set 
$\{I_t\wedge\tau<\infty\}$, see Proposition 7 of Chapter I in \cite{be}. Then from the construction of $Z^{(n)}$ 
and $Z$, 
\begin{eqnarray}
|Z^{(n)}_t-Z_t|&\le&\sup_{s\le I_t\wedge\tau}|X^{(n)}_{s}-X_{s}|
+|X_{I^{(n)}_t\wedge\tau^{(n)}}-X_{I_t\wedge\tau}|\,,
\end{eqnarray}
and from Theorem \ref{2065}, $\lim_nZ^{(n)}_t=Z_t$, $\p_x$-a.s.~on the set $\{I_t\wedge\tau<\infty\}$.
On the set $\{I_t\wedge\tau=\infty\}$, $\lim_n I^{(n)}_t\wedge\tau^{(n)}=\infty$, $\p_x$-a.s., so that 
$\lim_n Z^{(n)}_t=\lim_n X^{(n)}(I^{(n)}_t\wedge\tau^{(n)})=\infty=Z_t$, $\p_x$-a.s. The proof 
of the fact that $\lim_{n}\overline{Z}^{(n)}_t=\overline{Z}_t=\infty$, $\p_x$-a.s.~follows the same 
arguments together with the fact that 
$\sup_{s\le t}|\overline{X}^{(n)}_{s}-\overline{X}_{s}|=\overline{X}_{t}-\overline{X}^{(n)}_{t}$, 
so that $\lim_{n}\sup_{s\le t}|\overline{X}^{(n)}_{s}-\overline{X}_{s}|=0$, $\p$-a.s. for all $t\ge0$.

Let us now prove the second assertion. First observe that for all $0\leq m\le n$ and $t\geq 0$,
\begin{equation}\label{2964}
    \dfrac{X^{(n)}_t}{\overline{X}^{(n)}_t} \geq \dfrac{X^{(m)}_t}{\overline{X}^{(m)}_t}.
\end{equation}
Indeed, write $X^{(n)}_t=X^{(m)}_t + S^{m,n}_t$, where $S^{m,n}$ is a subordinator. Then 
$\overline{X}^{(n)}_t \leq \overline{X}^{(m)}_t+ S^{m,n}_t$, so that 
 \begin{align*}
    \dfrac{X^{(n)}_t}{\overline{X}^{(n)}_t} - \dfrac{X^{(m)}_t}{\overline{X}^{(m)}_t}
    =\dfrac{\overline{X}^{(m)}_tX_t^{(n)}-\overline{X}^{(n)}_tX^{(m)}_t}{\overline{X}^{(n)}_t\overline{X}^{(m)}_t}
    &=\dfrac{\overline{X}^{(m)}_t(X^{(m)}_t + S^{m,n}_t)- (\overline{X}^{(m)}_t+ S^{m,n}_t )X^{(m)}_t}
    {\overline{X}^{(n)}_t\overline{X}^{(m)}_t}\\
    &\geq \dfrac{ S^{m,n}_t (\overline{X}_t^{(m)}-X_t^{(m)})}{\overline{X}^{(n)}_t\overline{X}^{(m)}_t} \geq 0.
\end{align*}
Recall from Subsection \ref{9693} that we choose $\varepsilon_0$ small enough so that $0<\phi^{(n)'}(0)=
\phi'(\varepsilon_n)<\infty$, for all $n\ge0$, and hence $\e(X_1^{(n)})>0$, for all $n\ge0$. Let $\delta\in(0,1)$, 
then from Lemma 2.1 in \cite{lz}, there exists $t_0>0$ such that $\p$-a.s., for all $t\geq t_0$, 
$X^{(0)}_t/\overline{X}^{(0)}_t\geq 1-\delta$ and then from (\ref{2964}) for all $n\geq 0$, 
$X^{(n)}_t/\overline{X}^{(n)}_t\geq 1-\delta$. On the other hand, note that for all $t\ge0$, 
$\{\zeta\leq t\}\subset\{I_t\wedge\tau=\infty\}$ and hence  $\lim_n I^{(n)}_t\wedge\tau^{(n)}=\infty$, $\p$-a.s. 
on the set $\{\zeta\leq t\}$. This yields, 
\begin{equation*}
   \lim_n \dfrac{Z^{(n)}_t}{\overline{Z}_t^{(n)}} =\lim_n
     \dfrac{X^{(n)}(I^{(n)}_t\wedge\tau^{(n)})} {\overline{X}^{(n)}(I^{(n)}_t\wedge\tau^{(n)})} \geq 1-\delta\,,
     \;\;\;\mbox{$\p$-a.s. on the set $\{\zeta\leq t\}$,}
\end{equation*}
which allows us to conclude, $\delta$ being arbitrary.
\end{proof}

\begin{lemma}\label{999}
Let ${\rm e}$ be some exponentially distributed random variable with parameter $1$ and assume that ${\rm e}$ 
is independent of the family $(Z^{(n)},\,n\ge0)$. Fix $t>0$ and $x>0$. 
\begin{itemize}
\item[$1.$]
If $\lim_{n\to \infty}u^{(n)}_t(h(n))$ exists, then 
$\lim_{n\to\infty}u^{(n)}_t(kh(n))=\lim_{n\to \infty}u^{(n)}_t(h(n))$, for all $k>0$.
\item[$2.$] If $\lim_{n\rightarrow\infty}\p_x(\overline{Z}_t^{(n)}\le {\rm e}/h(n))=e^{-xu_t(0)}$, then 
$\lim_{n\rightarrow\infty}\p_x(Z_t^{(n)}\le {\rm e}/h(n))=e^{-xu_t(0)}$.
\item[$3.$] $\lim_{n\to \infty}u^{(n)}_t(h(n))$ exists if and only if 
$\lim_{n\rightarrow\infty}\p_x(\overline{Z}_t^{(n)}\le k{\rm e}/h(n))$ exists for all $k>0$ and does not 
depend on $k$. In this case, for all $k>0$, 
$\lim_{n\rightarrow\infty}\p_x(\overline{Z}_t^{(n)}\le k{\rm e}/h(n))=e^{-xl_t(h)}$, where 
$l_t(h):=\lim_{n\to \infty}u^{(n)}_t(h(n))$.
\end{itemize}
\end{lemma}
\begin{proof}
Let $\theta\in(0,\rho_0)$, so that none of the functions $\phi^{(n)}$ vanishes on $(0,\theta)$.  
Recall also that $h(0)<\rho_0$. Let $k>0$ and $n$ such that $kh(n)<\theta$ and write 
\[\int_{h(n)}^{kh(n)}\frac{du}{\phi^{(n)}(u)}=
\int_{h(n)}^\theta\frac{du}{\phi^{(n)}(u)}-
\int_{kh(n)}^\theta\frac{du}{\phi^{(n)}(u)}\,.\]
Since $\phi^{(n)}$ is increasing and differentiable, 
\begin{equation*}\label{Eq31}
\left| \int_{h(n)}^{kh(n)}\frac{du}{\phi^{(n)}(u)} \right|\le 
\frac{\left| k-1 \right| h(n)}{\phi^{(n)}((k\wedge 1)h(n))}=\frac{|k-1|}{(k\wedge 1)\phi'(\alpha_n)}\,,
\end{equation*}
where $\alpha_n\in({\varepsilon_n},{\varepsilon_n} +(k\wedge 1)h(n))$. Since $\phi'(0)=\infty$, it follows 
from the above inequality that 
$\lim_{n\rightarrow\infty}\int_{h(n)}^{kh(n)}\frac{du}{\phi^{(n)}(u)}=0$, 
so that if one of the sequences $(\int_{h(n)}^\theta\frac{du}{\phi^{(n)}(u)})$ or 
$(\int_{kh(n)}^\theta\frac{du}{\phi^{(n)}(u)})$ converges in $[0,\infty]$ as $n$ tends to infinity, 
then both converge and they have the same limit.
Then part $1.$ follows from part $3.$ of Lemma \ref{3420}.

Before proving $2.$ and $3.$, recall from Lemma \ref{CVPS} that $Z_t^{(n)}$ and $\overline{Z}_t^{(n)}$  
converge $\p_x$-almost surely to $Z_t$ and $\overline{Z}_t$, respectively. 
Therefore, for all $k>0$, 
\begin{equation}\label{9255}
\begin{array}{ll}
&\lim_{n\rightarrow\infty}\p_x(Z_t^{(n)}\le k{\rm e}/h(n),\,t<\zeta)=\p_x(Z_t<\infty,\,t<\zeta)=
\p_x(t<\zeta)=e^{-xu_t(0)}\,,\\
&\lim_{n\rightarrow\infty}\p_x(\overline{Z}_t^{(n)}\le k{\rm e}/h(n),\,t<\zeta)=
\p_x(\overline{Z}_t<\infty,\,t<\zeta)=\p_x(t<\zeta)=e^{-xu_t(0)}\,.
\end{array}
\end{equation}
Then let us prove assertion $2.$ and assume that 
$\lim_{n\rightarrow\infty}\p_x(\overline{Z}_t^{(n)}\le {\rm e}/h(n))=e^{-xu_t(0)}$. From (\ref{9255}) this 
means that $\lim_{n\rightarrow\infty}\p_x(\overline{Z}_t^{(n)}\le {\rm e}/h(n),\,t\ge\zeta)=0$.
Then let $a<1$ and write
\begin{eqnarray*}
\p_x\left(a^{-1}Z_t^{(n)}\le {\rm e}/h(n),\,t\ge\zeta\right)&\le&\p_x\left(\overline{Z}^{(n)}_t\le {\rm e}/h(n),\,
Z^{(n)}_t/\overline{Z}^{(n)}_t\in[a,1],\,t\ge\zeta\right)\\
&&+\p_x\left(Z^{(n)}_t/\overline{Z}^{(n)}_t\in[a,1]^c,t\ge\zeta\right)\,.
\end{eqnarray*}
This inequality together with Lemma \ref{CVPS} implies that 
$\lim_{n\rightarrow\infty}\p_x\left(a^{-1}Z_t^{(n)}\le {\rm e}/h(n),\,t\ge\zeta\right)=0$, so that from 
(\ref{9255}), $\lim_{n\rightarrow\infty}\p_x\left(a^{-1}Z_t^{(n)}\le {\rm e}/h(n)\right)=e^{-xu_t(0)}$.
But $\p_x\left(a^{-1}Z_t^{(n)}\le {\rm e}/h(n)\right)=\e_x(e^{-a^{-1}h(n)Z_t^{(n)}})=e^{-xu_t^{(n)}(a^{-1}h(n))}$, 
and the conclusion follows from part $1.$ of this lemma. 

Finally let us prove $3.$ Assume that $\lim_{n\rightarrow\infty}\p_x(\overline{Z}_t^{(n)}\le k{\rm e}/h(n))$ 
exists for all $k>0$ and does not depend on $k$.
Then from (\ref{9255}) $\lim_{n\rightarrow\infty}\p_x(\overline{Z}_t^{(n)}\le k{\rm e}/h(n),\,t\ge\zeta)$ exists 
for all $k>0$ and is equal to $\lim_{n\rightarrow\infty}\p_x(\overline{Z}_t^{(n)}\le {\rm e}/h(n),\,t\ge\zeta)$.
Now let $a>1$ and write
\begin{equation}\label{9543}
\p_x\left(aZ_t^{(n)}\le {\rm e}/h(n),\,\overline{Z}^{(n)}_t/Z^{(n)}_t
\in[1,a],t\ge\zeta\right)\le\p_x\left(\overline{Z}^{(n)}_t\le {\rm e}/h(n),\,t\ge\zeta\right)\,,
\end{equation}
so that from Lemma \ref{CVPS}, 
\begin{equation}\label{9544}
\limsup_{n\rightarrow\infty}\p_x\left(aZ_t^{(n)}\le {\rm e}/h(n),\,t\ge\zeta\right)\le 
\lim_{n\rightarrow\infty}\p_x\left(\overline{Z}^{(n)}_t\le {\rm e}/h(n),\,t\ge\zeta\right)\,.
\end{equation}
But since $\lim_{n\rightarrow\infty}\p_x(\overline{Z}_t^{(n)}\le k{\rm e}/h(n),\,t\ge\zeta)$ does not depend 
on $k$, replacing $h$ by $h/a$ in (\ref{9544}), we can write
\[\limsup_{n\rightarrow\infty}\p_x\left(Z_t^{(n)}\le {\rm e}/h(n),\,t\ge\zeta\right)\le 
\lim_{n\rightarrow\infty}\p_x\left(\overline{Z}^{(n)}_t\le {\rm e}/h(n),\,t\ge\zeta\right)\,.\]
Moreover, the inequality 
\[\lim_{n\rightarrow\infty}\p_x\left(\overline{Z}^{(n)}_t\le {\rm e}/h(n),\,t\ge\zeta\right)
\le\liminf_{n\rightarrow\infty}\p_x\left(Z^{(n)}_t\le {\rm e}/h(n),\,t\ge\zeta\right)\]
is straightforward. So we have proved that 
$\lim_{n\rightarrow\infty}\p_x\left(Z^{(n)}_t\le {\rm e}/h(n),\,t\ge\zeta\right)$ exists. 
We conclude from the equality, 
\begin{eqnarray}\label{1194}
\p_x\left(Z^{(n)}_t\le {\rm e}/h(n)\right)=e^{-xu_t^{(n)}(h(n))}\,.
\end{eqnarray}

Conversely assume that $\lim_{n\to \infty}u^{(n)}_t(h(n))$ exists. Then from part 1.~of the present lemma, 
$\lim_{n\to\infty}u^{(n)}_t(kh(n))=\lim_{n\to \infty}u^{(n)}_t(h(n))$, exists for all $k>0$ and therefore 
$\lim_n\p_x\left(Z^{(n)}_t\le {\rm e}/kh(n)\right)=\lim_ne^{-xu_t^{(n)}(kh(n))}$ exists for all $k>0$. Then 
we derive from  
(\ref{9255}) that $\lim_{n\rightarrow\infty}\p_x(Z_t^{(n)}\le k{\rm e}/h(n),\,t\ge\zeta)$ exists for all $k>0$
and does not depend on $k$. So from the same argument using Lemma \ref{CVPS}, developed in (\ref{9543}) and 
(\ref{9544}), for all $k>0$,
\begin{equation*}
\lim_{n\rightarrow\infty}\p_x\left(Z_t^{(n)}\le {\rm e}/h(n),\,t\ge\zeta\right)\le 
\liminf_{n\rightarrow\infty}\p_x\left(\overline{Z}^{(n)}_t\le k{\rm e}/h(n),\,t\ge\zeta\right)\,.
\end{equation*}
We conclude from the inequalities,
\begin{eqnarray*}
\limsup_{n\rightarrow\infty}\p_x\left(\overline{Z}^{(n)}_t\le k{\rm e}/h(n),\,t\ge\zeta\right)
&\le&\lim_{n\rightarrow\infty}\p_x\left(Z^{(n)}_t\le k{\rm e}/h(n),\,t\ge\zeta\right)\\
&=&\lim_{n\rightarrow\infty}\p_x\left(Z^{(n)}_t\le {\rm e}/h(n),\,t\ge\zeta\right)
\end{eqnarray*}
and the equality (\ref{1194}).
\end{proof}

\begin{lemma}\label{3233} Fix $x,t>0$, then
$\lim_{n\to \infty}u_t^{(n)}(h(n))$ exists if and only if 
$\lim_{n\rightarrow \infty}\p_x(\sigma_{k/h(n)}^{(n)}>t)$ exists for all $k>0$ and does not depend
on $k$. In this case, $\lim_{n\rightarrow \infty}\p_x(\sigma_{k/h(n)}^{(n)}>t)=e^{-xl_t(h)}$, where 
$\lim_{n\to \infty}u_t^{(n)}(h(n)):=l_t(h)$.
\end{lemma}
\begin{proof} Let ${\rm e}$ be as in Lemma $\ref{999}$ and note that 
$z\mapsto\p_x(\sigma_{z/h(n)}^{(n)}>t)$ is increasing. Then on the one hand, for all $\delta>0$, 
\begin{eqnarray}
\p_x(\sigma_{\delta/h(n)}^{(n)}>t)e^{-\delta}&\le&\int_\delta^\infty 
\p_x(\sigma_{z/h(n)}^{(n)}>t)e^{-z}\,dz\nonumber\\
&\le&\int_0^\infty 
\p_x(\sigma_{z/h(n)}^{(n)}>t)e^{-z}\,dz=
\p_x(\sigma_{{\rm e}/h(n)}^{(n)}>t)\,,\label{2564}
\end{eqnarray}
and on the other hand, for all $d>0$,
\begin{eqnarray}
\p_x(\sigma_{{\rm e}/h(n)}^{(n)}>t)&\le&
\p_x\left(\sigma_{{\rm e}/h(n)}^{(n)}>t,\,{\rm e}/h(n)<d/h(n)\right)+
\p_x\left({\rm e}/h(n)>d/h(n)\right)\nonumber\\
&=&\int_0^d\p_x(\sigma_{z/h(n)}^{(n)}>t)e^{-z}\,dz+\int_d^\infty e^{-z}\,dz\nonumber\\
&\le&(1-e^{-d})\p_x(\sigma_{d/h(n)}^{(n)}>t)+e^{-d}\,.\label{2565}
\end{eqnarray}

Then assume that $\lim_{n\to \infty}u_t^{(n)}(h(n))$ exists and set $\lim_{n\to \infty}u_t^{(n)}(h(n)):=l_t(h)$.
Recall that ${\rm e}$ is as in Lemma $\ref{999}$ and note that 
$\p_x(\sigma_{{\rm e}/h(n)}^{(n)}>t)=\p_x(\overline{Z}_t^{(n)}<{\rm e}/h(n))$. Then by part 3.~of Lemma 
\ref{999}, for all $k>0$,
\begin{equation}\label{2566}
\lim_{n\rightarrow\infty}\p_x(\sigma_{{\rm e}/h(n)}^{(n)}>t)=\lim_{n\rightarrow\infty}
\p_x(\sigma_{k{\rm e}/h(n)}^{(n)}>t)=e^{-xl_t(h)}\,.
\end{equation} 
Thanks to (\ref{2566}), replacing $h$ by $kh/\delta$  (resp.~by $kh/d$) in (\ref{2564}) 
(resp.~in (\ref{2565})), we obtain that for all $\delta,d,k>0$,  
\[e^{-\delta}\limsup_{n\rightarrow0}\p_x(\sigma_{k/h(n)}^{(n)}>t)\le e^{-xl_t(h)}\le 
(1-e^{-d})\liminf_{n\rightarrow0}\p_x(\sigma_{k/h(n)}^{(n)}>t)+e^{-d}\,.\]
Taking $\delta$ to 0 and $d$ to $\infty$ shows that $\lim_{n\to \infty}\p_x(\sigma_{k/h(n)}^{(n)}>t)=e^{-xl_t(h)}$.

Conversely if $\lim_{n\rightarrow \infty}\p_x(\sigma_{k/h(n)}^{(n)}>t)$ exists for all $k>0$
and does not depend on $k$, then replacing $h$ by $h/k$ in (\ref{2564}) and (\ref{2565}), we obtain,
\begin{eqnarray*}
&&\lim_{n\rightarrow \infty}\p_x(\sigma_{1/h(n)}^{(n)}>t)e^{-\delta}\,
\le\liminf_{n\rightarrow\infty}\p_x(\sigma_{k{\rm e}/h(n)}^{(n)}>t)\le
\limsup_{n\rightarrow\infty}\p_x(\sigma_{k{\rm e}/h(n)}^{(n)}>t)\\
&&\le(1-e^{-d})\lim_{n\rightarrow \infty}\p_x(\sigma_{1/h(n)}^{(n)}>t)+e^{-d}\,.
\end{eqnarray*}
Taking $\delta$ to 0 and $d$ to $\infty$ shows that 
$\lim_{n\rightarrow\infty}\p_x(\sigma_{k{\rm e}/h(n)}^{(n)}>t)$ exists for all $k>0$ and does not 
depend on $k$. We conclude from part $2.$ of Lemma $\ref{999}$.
\end{proof}

\noindent {\it Proof of Theorem $\ref{3332}$}. Equivalence between $1.$ and $3.$ is given by Lemma 
\ref{3233}. Equivalence between $2.$ and $3.$ is given by $3.$ in Lemma \ref{3420}. Then still from 
$3.$ in Lemma \ref{3420} and Lemma \ref{3233}, for all $x,t\ge0$,
$\lim_{n\rightarrow\infty}\p_x(\sigma_{1/h(n)}^{(n)}>t)=\lim_{n\rightarrow\infty}
e^{-xu_t^{(n)}(h(n))}=e^{-u_{(t-c(h))^+}(0)}=\p_x(\zeta + c(h) > t)$, which achieves the 
proof of the theorem. $\hfill\Box$\\

\noindent {\it Proof of Proposition $\ref{0573}$}.  This proof follows from $2.$ and $3.$ in Lemma 
\ref{999} and the fact that $\p_x(Z_t^{(n)}\le {\rm e}/h(n))=e^{-xu_t^{(n)}(h(n))}$. $\hfill\Box$\\

\noindent {\it Proof of Proposition $\ref{1442}$}. 
Assume that 
$\sigma^{(n)}_{1/h(n)} \xrightarrow[n\to\infty]{(d)} \zeta$, that is, for all 
$t>0$,  $\lim_{n\to \infty}\p_x( \overline{Z}_t^{(n)} \leq 1/h(n) )=e^{-xu_t(0)}$. 
Note that we can replace the exponential time ${\rm e}$ by $1$ in (\ref{9255}) so that for 
all decreasing sequence $(f(n))$ such that $\lim_{n\rightarrow\infty}f(n)=0$,
\begin{equation}\label{1256}
\lim_{n\rightarrow\infty}\p_x(\overline{Z}_t^{(n)}\le 1/f(n),\,t<\zeta)=
\p_x(\overline{Z}_t<\infty,\,t<\zeta)=\p_x(t<\zeta)=e^{-xu_t(0)}.
\end{equation}
From the assumption and (\ref{1256}) applied to $f=h$, 
$\lim_{n\rightarrow\infty}\p_x(\overline{Z}_t^{(n)}\le 1/h(n),\,t\geq\zeta)= 0$.
Moreover, since $h(n)\le\tilde{h}(n)$, for all $n$ large enough,
\begin{equation*}
\p_x(\overline{Z}_t^{(n)}\le 1/\tilde{h}(n),\,t\geq\zeta) \leq 
\p_x(\overline{Z}_t^{(n)}\le 1/h(n),\,t\geq\zeta)
\end{equation*}
so that the left hand side of this inequality tends to $0$ when $n$ goes to infinity. Applying this 
together with (\ref{1256}) for $f=\tilde{h}$ gives
$\sigma^{(n)}_{1/\tilde{h}(n)} \xrightarrow[n\to\infty]{(d)} \zeta$. $\hfill\Box$\\ 

\noindent {\it Proof of Proposition $\ref{1332}$}. 
Let $(\tilde{h(n)})$ be such a sequence. Then there are $k_1,k_2>0$ such that for all $n$ sufficiently 
large, $k_1 h(n)\le \tilde{h}(n)\le k_2 h(n)$, so that from $4.$ in Lemma \ref{9522}, for all $t\ge0$,
$u^{(n)}_t(k_1 h(n))\le u^{(n)}_t( \tilde{h}(n))\le u^{(n)}_t(k_2 h(n))$. Then assertion $1.$ follows 
from Theorem \ref{3332}.
To prove the second assertion it suffices to note that 
$\lim_{n\rightarrow\infty}\int_{h(n)}^{\tilde{h}(n)}\frac{du}{\phi^{(n)}(u)}=c(h)-c(\tilde{h})\ge0$.
Then to prove part $3.$ let us write, for $\theta\in(0,\rho)$,
\begin{equation}\label{2577}
\int_{h(n)}^{\theta}\frac{du}{\phi^{(n)}(u)}-\int_{h(n)}^{\tilde{h}(n)}\frac{du}{\phi^{(n)}(u)}
=\int_{\tilde{h}(n)}^\theta\frac{du}{\phi^{(n)}(u)}\,,
\end{equation}
so that taking the liminf on each side gives,
\[\lim_{n\rightarrow\infty}\int_{h(n)}^{\theta}\frac{du}{\phi^{(n)}(u)}-
\limsup_{n\rightarrow\infty}\int_{h(n)}^{\tilde{h}(n)}\frac{du}{\phi^{(n)}(u)}
=\liminf_{n\rightarrow\infty}\int_{\tilde{h}(n)}^\theta\frac{du}{\phi^{(n)}(u)}\,.\]
But $h\in\mathcal{Z}_0$, so that
$\lim_{n\rightarrow\infty}\int_{h(n)}^{\theta}\frac{du}{\phi^{(n)}(u)}=
\int_{0}^{\theta}\frac{du}{\phi(u)}$, moreover from Fatou's lemma,\\ 
$\liminf_{n\rightarrow\infty}\int_{\tilde{h}(n)}^\theta\frac{du}{\phi^{(n)}(u)}\ge
\int_{0}^{\theta}\frac{du}{\phi(u)}$, which shows that 
$\lim_{n\rightarrow\infty}\int_{h(n)}^{\tilde{h}(n)}\frac{du}{\phi^{(n)}(u)}=0$ 
and finally from (\ref{2577}),
$\lim_{n\rightarrow\infty}\int_{\tilde{h}(n)}^\theta\frac{du}{\phi^{(n)}(u)}=
\int_{0}^{\theta}\frac{du}{\phi(u)}$. We conclude from Theorem \ref{3332}.
For the last part, note that since $c(h)=\infty$, for all $t\ge0$, $l_t(h)=0$. Then recall from 
Lemma \ref{9522} that $\lambda\mapsto u_t^{(n)}(\lambda)$ is increasing so that for all $n$,
$u_t^{(n)}(\tilde{h}(n))\le u_t^{(n)}(h(n))$, which implies that for all $t\ge0$, $l_t(\tilde{h})=0$
so that $\tilde{h}\in Z_\infty$. $\hfill\Box$\\ 

Let us note the following inequalities which are direct consequences of the concavity of $\phi$. 
Let $\gamma$ be the unique value such that 
$\gamma\in(0,\rho)$ and $\phi'(\gamma)=0$ if $X$ is not a subordinator and $\gamma=\infty$ otherwise. 
Then for all $u\in(0,\gamma)$ and $n\ge0$ such that $u+\varepsilon_n\in(0,\gamma)$,
$\phi'(u+{\varepsilon_n})  \leq\dfrac{\phi^{(n)}(u)}{u} \leq \phi'({\varepsilon_n})$, so that
\begin{equation} \label{concav}
\dfrac{1}{u\phi'({\varepsilon_n})}\leq\dfrac{1}{\phi^{(n)}(u)}\leq
\dfrac{1}{u\phi'(u+{\varepsilon_n})}\,.
\end{equation}

\begin{lemma} \label{RV1} $\mbox{}$
\begin{itemize}
\item[$1.$] If $\displaystyle\lim_{n\to\infty} \int_{h(n)}^\theta \dfrac{du}{\phi^{(n)}(u)} = 
\int_0^\theta \dfrac{du}{\phi(u)}$, for some $\theta\in(0,\rho)$, then
$\lim_{n\to\infty}\dfrac{|\log{h({n})}|}{\phi'({\varepsilon_n})}=0$.
\item[$2.$] If $\displaystyle\lim_{n\to\infty}\dfrac{|\log{h(n)}|}{\phi'({\varepsilon_n})}=\infty$,
then $\displaystyle\lim_{n\to\infty} \int_{h(n)}^\theta \dfrac{du}{\phi^{(n)}(u)} = \infty$
for all $\theta\in(0,\rho)$.
\end{itemize}
\end{lemma}
\begin{proof}
From the assumption and dominated convergence on $[a,\theta)$, we obtain that for all $a\in(0,\theta)$,
    \begin{equation*}
        \lim_{n\to\infty}\int_{h(n)}^a \dfrac{du}{\phi^{(n)}(u)}=\int_0^a \dfrac{du}{\phi(u)}.
    \end{equation*}
On the other hand recall that $\lim_{n\rightarrow\infty}\phi'(\varepsilon_n)=\infty$. Then for all 
$a\in(0,\gamma)$ and $h(n)\le a$, (\ref{concav}) yields,
\begin{equation*}
    \int_{h(n)}^a \dfrac{du}{\phi^{(n)}(u)}\geq \dfrac{\log(a)}{\phi'({\varepsilon_n})} - 
    \dfrac{\log(h(n))}{\phi'({\varepsilon_n})} \implies \limsup_{n\to\infty} 
    \dfrac{|\log(h(n))|}{\phi'({\varepsilon_n})} \leq \int_0^a \dfrac{du}{\phi(u)}. 
\end{equation*}
The first assertion is obtained by taking the limit  when $a$ tends to $0$ in the last inequality.
Then second assertion is a direct consequence of (\ref{concav}).
\end{proof}

\noindent {\it Proof of Proposition $\ref{1100}$}. It follows directly from 
Lemma \ref{RV1} and the definitions of $\mathcal{Z}$ and $\mathcal{Z}_0$. $\hfill\Box$\\

\noindent {\it Proof of Theorem $\ref{2385}$}. This proof requires the following additional lemma.

\begin{lemma}\label{9331} If $\liminf_{n\rightarrow\infty}\phi(h(n)+\varepsilon_n)/\phi(\varepsilon_n)>1$, 
then for all $\theta\in(0,\rho)$,
\begin{equation*}
    \lim_{n\to\infty} \int_{h(n)}^\theta \dfrac{du}{\phi^{(n)}(u)} = \int_0^\theta 
    \dfrac{du}{\phi(u)},
\end{equation*}
that is $h\in\mathcal{Z}_0$.
\end{lemma}
\begin{proof} We can assume with no loss of generality that $\theta\in(0,\gamma)$.
Then we derive from the monotone convergence theorem that 
$\displaystyle \lim_{n\to\infty}\int_{h(n)}^\theta 
    \dfrac{du}{\phi(u+\varepsilon_n)}=\int_0^\theta \dfrac{du}{\phi(u)}$.
On the other hand, for all $n\ge0$ such that $u+\varepsilon_n\in(0,\theta)$,  
\begin{align*}
    0\leq \dfrac{1}{\phi^{(n)}(u)}-\dfrac{1}{\phi(u+\varepsilon_n)}=
    \dfrac{\phi(\varepsilon_n)}{\phi^{(n)}(u)\phi(u+\varepsilon_n)}
    \leq \dfrac{\phi(\varepsilon_n)}{\phi^{(n)}(u)\phi(u)}.
\end{align*}
Moreover, for $u\geq h(n)$,
\begin{equation}\label{9332}
    \dfrac{\phi(\varepsilon_n)}{\phi^{(n)}(u)}\leq
    \dfrac{\phi(\varepsilon_n)}{\phi(h(n)+\varepsilon_n)-\phi(\varepsilon_n)} = 
    \dfrac{1}{\phi(h(n)+\varepsilon_n)/\phi(\varepsilon_n)-1}
\end{equation}
and from our assumption,  there exists $C>0$ such that for all $n$ large enough
\begin{equation*}
    \mathds{1}_{[h(n),\theta]}(u)\left(\dfrac{1}{\phi^{(n)}(u)}-
    \dfrac{1}{\phi(u+\varepsilon_n)}\right)\leq \dfrac{C}{\phi(u)}\,.
\end{equation*}
Then the result follows from dominated convergence.
\end{proof}

We are now ready to achieve the proof of Theorem \ref{2385}.
Assume now that $\phi=-\varphi$ is regularly varying at 0, with index $\alpha\in(0,1)$. 
Then from Theorem 1.7.2b in \cite{bgt}, $\phi'$ is regularly varying at 0, with index $\alpha-1$. \\

We assume first that for all $k>0$, $\sigma^{(n)}_{k/h(n)} \xrightarrow[n\to\infty]{(d)} \zeta + c(h)$, 
that is $h\in\mathcal{Z}$. If $c(h)=0$ then by Proposition \ref{1100}, 
$\lim_{n\to\infty}\dfrac{|\log h(n)|}{\phi'(\varepsilon_n)} = 0$. Now assume that $c(h)>0$, let $d>0$ and 
write,
\begin{equation}\label{2399}
\int_{h(n)}^\theta \dfrac{du}{\phi^{(n)}(u)} =
\int_{h(n)}^{d\varepsilon_n} \dfrac{du}{\phi^{(n)}(u)} + 
\int_{d\varepsilon_n}^\theta 
\dfrac{du}{\phi^{(n)}(u)}.
\end{equation}
Since $\lim_n\phi((1+d)\varepsilon_n)/\phi(\varepsilon_n)=(1+d)^\alpha>1$, Lemma \ref{9331} implies that
the second term of the right hand side of (\ref{2399}) tends to 
$\displaystyle\int_0^\theta\dfrac{du}{\phi(u)}$ as $n\rightarrow\infty$, that is, 
$n\mapsto d\varepsilon_n \in\mathcal{Z}_0$. It follows from Theorem \ref{3332} that
\begin{equation*}
\lim_{n\to\infty}\int_{h(n)}^{d{\varepsilon_n}} \dfrac{du}{\phi^{(n)}(u)}=  
\lim_{n\to\infty} \left(\int_{h(n)}^\theta \dfrac{du}{\phi^{(n)}(u)}-
\int_{d\varepsilon_n}^\theta \dfrac{du}{\phi^{(n)}(u)}\right) = c(h)\,.
\end{equation*}
Since $c(h)>0$, for all $n$ large enough, $h(n)<d\varepsilon_n$, and then the inequalities 
(\ref{concav}) yield, 
\begin{equation}\label{4676}
\dfrac{\log{d{\varepsilon_n}}-\log{h(n)}}{\phi'( {\varepsilon_n} )}\leq 
\int_{h(n)}^{d{\varepsilon_n}} \dfrac{du}{\phi^{(n)}(u)}\leq 
\dfrac{\log{d{\varepsilon_n}}-\log{h(n)}}{\phi'(d{\varepsilon_n} +{\varepsilon_n} )}.
\end{equation}
Moreover, since $\phi'$ is regularly varying,
\begin{equation*} 
\limsup_{n\to\infty} \dfrac{-\log{h(n)}}{\phi'({\varepsilon_n})}  \leq c(h) \leq 
(d+1)^{1-\alpha}  \liminf_{n\to\infty} 
\dfrac{-\log{h(n)}}{\phi'({\varepsilon_n})}.
\end{equation*}
We conclude that $2.$ implies $1.$ by taking the limit in this inequality as $d$ tends to $0$.

Now suppose that $c(h) = \displaystyle\lim_{n\to\infty}\dfrac{|\log{h(n)}|}{\phi'({\varepsilon_n})}$ 
exists. Assume first that $c(h)>0$ and recall that since $\phi'$ is regularly varying with index 
$\alpha-1$, $\displaystyle\lim_{n\to\infty} \dfrac{|\log{\varepsilon_n}|}{\phi'({\varepsilon_n})} = 0$. 
This implies that for all $d>0$ and for all $n$ large enough, $h(n)<d \varepsilon_n$. Using 
(\ref{2399}) and (\ref{4676}) and taking the limit when $d$ goes to $0$ yields
$\displaystyle\lim_{n\to\infty} \int_{h(n)}^\theta \dfrac{du}{\phi^{(n)}(u)} 
=\int_0^\theta \dfrac{du}{\phi(u)} + c(h)$. Then Theorem \ref{3332} allows us to derive that for all 
$k>0$, $\sigma^{(n)}_{k/h(n)} \xrightarrow[n\to\infty]{(d)} \zeta + c(h)$. If $c(h) = 0$, then 
define the two subsequences $(n_i)$ and $(m_i)$ such that for all $i\ge0$, 
$h(n_i)<\varepsilon_{n_i}$ and $h(m_i)\geq\varepsilon_{m_i}$.
For the subsequence $(n_i)$, inequalities (\ref{2399}) and (\ref{4676}) with $d=1$ 
allow us to obtain that $\displaystyle\lim_n\int_{h(n)}^\theta \dfrac{du}{\phi^{(n)}(u)}=
\int_{0}^\theta \dfrac{du}{\phi(u)}$. 
We conclude from Theorem \ref{3332} that 2.~holds for the subequence $(n_i)$.
For the subsequence $(m_i)_{i\in\mathbb{N}}$, since, as already noticed above,  
$m_i\mapsto\varepsilon_{m_i}\in\mathcal{Z}_0$, it suffices to apply part $3.$ of Proposition 
$\ref{1332}$. Then we have proved that 1.~and 2.~are equivalent in Theorem \ref{2385}.\\

It is enough to prove that 3~implies 1. Assume  without loss of generality that 3.~holds for $k=1$. 
Since for all $k>0$ and all $n$ 
sufficiently large, $kh(n)\le \sqrt{h(n)}$, Proposition \ref{1442} implies that 
$\sigma^{(n)}_{k/\sqrt{h(n)}} \xrightarrow[n\to\infty]{(d)} \zeta$, for all $k>0$. Therefore, from 
the equivalence between parts 1.~and 2.~of Theorem \ref{2385}, $\displaystyle\lim_{n\to\infty}
\dfrac{|\log{\sqrt{h(n)}}|}{\phi'({\varepsilon_n})}=0$ and hence 
$\displaystyle\lim_{n\to\infty}\dfrac{|\log{{h(n)}}|}{\phi'({\varepsilon_n})}=0$, which is part 1.~of
Theorem \ref{2385}. $\hfill\Box$\\ 

\subsection{Proof of the results in Subsection $\ref{8347}$}

\subsubsection{Convergence in $L_1$}\label{5935}

\noindent {\it Proof of Proposition $\ref{8453}$}. Denote by $\p_x^\uparrow$ the probability under 
which $X$ starts from $x$ and is conditioned to never reach $0$. 
Recall that $\p_x(Z_t<\infty) = \p_x(\zeta>t) = e^{-xu_t(0)}$ and the notation 
$\tau=\inf\{t\ge0:X_t\le0\}$. Note also that for all $t\ge0$, $\{\tau<\infty\}\subset\{Z_t<\infty\}$
and that for all $x>0$, $\p_x(\tau<\infty)=e^{-\rho x}$. Then we can write
\begin{align*}
    \p_x^\uparrow( \zeta>t ) &= \p_x( Z_t<\infty , \tau = \infty) /\p_x( \tau = \infty) \\
    &=  \left( \p_x( Z_t<\infty) -\p_x( Z_t<\infty, \tau < \infty) \right) /\p_x( \tau = \infty)
    = \frac{ e^{-xu_t(0)} - e^{-\rho x} }{ 1-e^{-\rho x}}\,.
\end{align*}
Therefore, from (\ref{3451}), under $\p_x^\uparrow$, the law of $\zeta$ is absolutely continuous on 
$[0,\infty)$ with respect to the Lebesgue measure and its density is given by 
$f_x(t):=\frac{x\phi(u_t(0))e^{-x u_t(0)}}{1-e^{-\rho x}}$. This allows us to write
    \begin{equation*}
         \e_x^\uparrow(\zeta^n\1{\zeta<\infty}) =  
         \int_0^\infty t^n \frac{x\phi(u_t(0))e^{-x u_t(0)}}{1-e^{-\rho x}}dt\,.
    \end{equation*}
The result follows using the substitution $u_t(0)=\lambda$ and the fact that $u_t(0)=F^{-1}(t)$, where 
$F:y\mapsto\int_0^ydu/\phi(u)$ 
is a bijection from $[0,\rho)$ to $[0,\infty)$,  together with the equality $ \e_x(\zeta^n\1{\zeta<\infty}) 
=(1-e^{-\rho x})\e_x^\uparrow(\zeta^n\1{\zeta<\infty})$. Since $|\phi'(\rho)|\in(0,\infty)$, 
$\phi(u)\sim(\rho-u)|\phi'(\rho)|$ as $u\uparrow\rho$, so that $F(y) =_{y\to \rho} O(\log(\rho-y))$, which 
implies that the moments of $\zeta$ are finite. 

We show the second equality from same arguments.
$\hfill\Box$\\

Recall from Subsection \ref{9693} the definition of $(\varepsilon_n)$, $(h(n))$, 
$(X^{(n)})$ and $X$. Let us also emphasize that 
\[h\in\mathcal{Z}_0\,,\]
throughout the rest of this subsection. Moreover, in this subsection, we assume that $X$ is not a 
subordinator, in particular $\rho\in(0,\infty)$. Indeed, the proof of Theorem \ref{0486} is actually 
simpler in the case  where $X$ is a subordinator and we allow ourself to skip it. We also recall the 
notation $\tau=\inf\{t:X_t\le0\}$ and introduce the following one,
\[\tau^{(n)}=\inf\{t:X_t^{(n)}\le0\}\,.\]

\begin{lemma}\label{3722}
Let $\psi:[0,\infty)\rightarrow[\rho,\infty)$ be the inverse of $\varphi$, that is, 
$\varphi\circ\psi(\lambda)=\lambda$ and define $\psi^{(n)}:[0,\infty)\rightarrow[\rho_n,\infty)$ 
the inverse of $\varphi^{(n)}$ in the same way, where $\rho_n$ is the largest root of 
$\varphi^{(n)}$. Let ${\rm e}_n:={\rm e}/h(n)$ where ${\rm e}$ is an exponentially distributed random 
variable with parameter $1$ which is independent of $(X^{(n)})$ and $X$. Then
\begin{equation}\label{4303}
\e_x\left(\int_0^{{\rm e}_n}\frac{du}{X_u}\ind_{\{\tau=\infty\}}\right) = 
\int_0^\rho \frac{e^{-\lambda x}-e^{\psi(h(n))x}}{h(n)+\phi(\lambda)} d\lambda \,,
\end{equation}
\begin{equation}\label{4066}
\e_x\left(\int_0^{{\rm e}_n}\frac{du}{X^{(n)}_u}\ind_{\{\tau^{(n)}=\infty\}}\right) = 
\int_0^{\rho_n}\frac{e^{-\lambda x}-e^{-\psi^{(n)}(h(n))x}}{h(n)+\phi^{(n)}(\lambda)} 
d\lambda \,.
\end{equation}
\end{lemma}
\begin{proof}
Let us set $\underline{X}_u=\inf_{t\le u}X_t$, then
\begin{eqnarray*}
\e_x(e^{-\lambda X_u}\ind_{\{\tau=\infty\}})&=&e^{-\lambda x}E(e^{-\lambda X_u}
\ind_{\{\underline{X}_u\ge-x\}}P_{X_u}(\tau_{-x}=\infty))\\
&=&e^{-\lambda x}E(e^{-\lambda X_u}\ind_{\{\underline{X}_u\ge-x\}}P_{X_u}(-\underline{X}_\infty\le x))\\
&=&e^{-\lambda x}E(e^{-\lambda X_u}\ind_{\{\underline{X}_u\ge-x\}}(1-e^{-\rho(X_u+x)}))\,,
\end{eqnarray*} 
since $-\underline{X}_\infty$ is exponentially distributed with parameter $\rho$. Let $\varepsilon>0$ 
and set ${\tt e}:={\rm e}/\varepsilon$, so that ${\tt e}$ is an exponentially distributed r.v.~with parameter 
$\varepsilon$ which is independent of $X$. Then according to the previous equality, 
\begin{eqnarray*}
\e_x(e^{-\lambda X_{\tt e}}\ind_{\{\tau=\infty\}})&=&e^{-\lambda x}\e(e^{-\lambda X_{\tt e}}
\ind_{\{\underline{X}_{\tt e}\ge-x\}}(1-e^{-\rho(X_{\tt e}+x)}))\\
&=&e^{-\lambda x}\e(e^{-\lambda X_{\tt e}}\ind_{\{\underline{X}_{\tt e}\ge-x\}})-
e^{-(\lambda+\rho)x}\e(e^{-(\lambda+\rho)X_{\tt e}}\ind_{\{\underline{X}_{\tt e}\ge-x\}})\\
&=&e^{-\lambda x}\e(e^{-\lambda (X_{\tt e}-\underline{X}_{\tt e}}))\e(e^{-\lambda \underline{X}_{\tt e}}
\ind_{\{\underline{X}_{\tt e}\ge-x\}})-\\
&&
e^{-(\lambda+\rho)x}\e(e^{-(\lambda+\rho)(X_{\tt e}-\underline{X}_{\tt e})})
\e(e^{-(\lambda+\rho)\underline{X}_{\tt e}}\ind_{\{\underline{X}_{\tt e}\ge-x\}})\,,
\end{eqnarray*} 
where the third equality above comes from the fact that $\underline{X}_{\tt e}$ and 
$X_{\tt e}-\underline{X}_{\tt e}$ are independent, see Th. 5, chap. VI in \cite{be}. 
On the other hand recall that $-\underline{X}_{\tt e}$ is exponentially distributed with parameter
$\psi(\varepsilon)$ and that the law of $X_{\tt e}-\underline{X}_{\tt e}$ is given by
\[\e(e^{-\alpha(X_{\tt e}-\underline{X}_{\tt e})})=\frac{\varepsilon}{\varepsilon-\varphi(\alpha)}\left(1-
\frac{\alpha}{\psi(\varepsilon)}\right)\,,\]
see Th. 4, chap. VII in \cite{be}. We derive from above that,
\begin{eqnarray}
\e_x(e^{-\lambda X_{\tt e}}\ind_{\{\tau=\infty\}})&=&e^{-\lambda x}
\frac{\varepsilon}{\varphi(\lambda)-\varepsilon}(e^{(\lambda-\psi(\varepsilon))x}-1)-\label{2506}\\
&&e^{-(\lambda+\rho) x}\frac{\varepsilon}{\varphi(\lambda+\rho)-\varepsilon}
(e^{((\lambda+\rho)-\psi(\varepsilon))x}-1)\,.\nonumber
\end{eqnarray} 
On the other hand, note that 
\begin{eqnarray*}
\e_x\left(\int_0^{{\tt e}}\frac{du}{X_u}\ind_{\{\tau=\infty\}}\right)&=&
\e_x\left(\int_0^\infty\varepsilon e^{-\varepsilon y}dy\int_0^y\frac{du}{X_u}\ind_{\{\tau=\infty\}}\right)\\
&=&\e_x\left(\varepsilon\int_0^\infty\frac{du}{X_u}\int_u^\infty 
e^{-\varepsilon y}\,dy\ind_{\{\tau=\infty\}}\right)\\
&=&\e_x\left(\int_0^\infty\frac{e^{-\varepsilon u}}{X_u}\,du\ind_{\{\tau=\infty\}}\right)
=\e_x\left(\frac{1}{\varepsilon X_{{\tt e}}}\ind_{\{\tau=\infty\}}\right)\,.
\end{eqnarray*}
Applying this remark and taking $\varepsilon=h(n)$ in the relation (\ref{2506}) once integrated over 
$(0,\infty)$ with respect to $\lambda$, yields
\begin{eqnarray}\e_x\left(\int_0^{{\rm e}_n}\frac{du}{X_u}\ind_{\{\tau=\infty\}}\right)=
\int_0^\infty 
\frac{e^{-\psi(h(n))x}-e^{-\lambda x}}{\varphi(\lambda)-h(n)}-
\frac{e^{-\psi(h(n))x}-e^{-(\lambda+\rho) x}}{\varphi(\lambda+\rho)-h(n)}\,d\lambda,\label{380}
\end{eqnarray}
where ${\rm e}_n$ is defined in the statement. 
We can check that the function 
$\displaystyle f(\lambda):=\frac{e^{-\psi(h(n))x}-e^{-\lambda x}}{\varphi(\lambda)-h(n)}$ is integrable 
over $(0,M)$, for all $M>0$ so that 
\begin{eqnarray*}
\int_0^\infty f(\lambda)-f(\lambda+\rho)\,d\lambda
&=&\lim_{M\to \infty } \int_0^M f(\lambda )d\lambda  - \int_0^M f(\lambda + \rho) d\lambda \\
&=&\lim_{M\to \infty } \left(\int_0^\rho f(\lambda )d\lambda - \int_M^{M+\rho} f(\lambda) d\lambda\right)=
\int_0^\rho f(\lambda )d\lambda\,.
\end{eqnarray*}
This gives (\ref{4303}) and applying the same arguments to $X^{(n)}$ gives (\ref{4066}). 
\end{proof}

\begin{lemma}\label{3872}
Let ${\rm e}_n$ be as in Lemma $\ref{3722}$. Then for all $x>0$, 
\begin{eqnarray}
&&\lim_{n\rightarrow\infty}\e_x\left(\int_0^{{\rm e}_n}\frac{du}{X_u}\ind_{\{\tau=\infty\}}\right)= 
\lim_{n\rightarrow\infty}\e_x\left(\int_0^{{\rm e}_n}\frac{du}{X^{(n)}_u}\ind_{\{\tau^{(n)}=\infty\}}\right)\nonumber\\
&=&\e_x\left(\int_0^{\infty}\frac{du}{X_u}\ind_{\{\tau=\infty\}}\right)=
\int_0^{\rho}\frac{e^{-\lambda x}-e^{-\rho x}}{\phi(\lambda)}<\infty\,.\label{9280}
\end{eqnarray}
\end{lemma}
\begin{proof}
Fix $x>0$, then the equality 
$\lim_{n\rightarrow\infty}\e_x\left(\int_0^{{\rm e}_n}\frac{du}{X_u}\ind_{\{\tau=\infty\}}\right)=
\e_x\left(\int_0^{\infty}\frac{du}{X_u}\ind_{\{\tau=\infty\}}\right)$ simply follows from monotone convergence. 
For the other equalities, thanks to Lemma \ref{3722} it suffices to show that, 
\begin{eqnarray}
&&\lim_{n\rightarrow\infty}\int_0^\rho \frac{e^{-\lambda x}-e^{-\psi(h(n))x}}{h(n)+\phi(\lambda)} d\lambda=
\lim_{n\rightarrow\infty}\int_0^{\rho_n} \frac{e^{-\lambda x}-e^{-\psi^{(n)}(h(n))x}}{h(n)+\phi^{(n)}(\lambda)} 
d\lambda\nonumber\\
&&=\int_0^{\rho}\frac{e^{-\lambda x}-e^{-\rho x}}{\phi(\lambda)}d\lambda<\infty\,.\label{3775}
\end{eqnarray}
We shall only prove that
$\lim_{n\rightarrow\infty}\int_0^{\rho_n} \frac{e^{-\lambda x}-e^{-\psi^{(n)}(h(n))x}}
{h(n)+\phi^{(n)}(\lambda)}d\lambda=\int_0^{\rho}\frac{e^{-\lambda x}-e^{-\rho x}}{\phi(\lambda)}d\lambda<\infty$.
The first limit follows the same lines and is actually easier to obtain. 
Recall that for all $n$, $\rho_n\le\psi^{(n)}(h(n))$, $\rho_n\le\rho$ and that 
both sequences $(\psi^{(n)}(h(n)))$ and $(\rho_n)$ converge to $\rho$. Then write, 
\begin{eqnarray}
\int_0^{\rho_n} \frac{e^{-\lambda x}-e^{-\psi^{(n)}(h(n))x}}{h(n)+\phi^{(n)}(\lambda)} d\lambda&=&
\int_{0}^{\rho_n}\frac{e^{-\lambda x}-e^{-\rho_n x}}{h(n)+\phi^{(n)}(\lambda)}d\lambda\label{7753}\\
&&+(e^{-\rho_n x}-e^{-\psi^{(n)}(h(n))x})\int_0^{\rho_n}\frac{ d\lambda}{h(n)+\phi^{(n)}(\lambda)}.\nonumber
\end{eqnarray}
Note that $\phi^{(n)}(\rho_n)=0$ and hence from the mean value theorem, for 
all $n\ge0$ and $\lambda\in(0,\rho_n)$, there is $\alpha_n\in(\lambda,\rho_n)$ such that 
$\phi^{(n)}(\lambda)=(\rho_n-\lambda)|\phi^{(n)'}(\alpha_n)|$. But recall that $(\rho_n)$ is 
increasing and that $\lim_n \rho_n=\rho$. Since 
$\lim_{n\rightarrow\infty}\phi^{(n)'}=\phi'$ and $\phi'(\rho)\in(0,\infty)$, there are $c>0$, 
$\alpha\in(0,\rho_0)$ and $n_0\ge0$, such that for all $n\ge n_0$ and $\lambda\in(\rho_n-\alpha,\rho_n)$, 
$|\phi^{(n)'}(\lambda)|\ge c$. This implies that for all $n\ge n_0$ and all $\lambda\in(\rho_n-\alpha,\rho_n)$,
\begin{equation}\label{7501}
\frac{\rho_n-\lambda}{\phi^{(n)}(\lambda)}\le c^{-1}\,,
\end{equation}
so that,
\begin{equation}\label{9910}
\int_{\rho_n-\alpha}^{\rho_n}\frac{ d\lambda}{h(n)+\phi^{(n)}(\lambda)}=O(\log(h(n)))\,.
\end{equation}
Moreover, for all $\delta\in(0,1)$ and $n$ sufficiently large,
\begin{eqnarray*}
\int_0^{\rho_n-\alpha} \frac{d\lambda}{h(n)+\phi^{(n)}(\lambda)}&=&
\int_0^{\delta h(n)} \frac{d\lambda}{h(n)+\phi^{(n)}(\lambda)}+
\int_{\delta h(n)}^{\rho_n-\alpha} \frac{d\lambda}{\phi^{(n)}(\lambda)}\\
&\le&(\mbox{cst})\delta+\int_{\delta h(n)}^{\rho_n-\alpha} \frac{d\lambda}{\phi^{(n)}(\lambda)}\,,
\end{eqnarray*}
and $\int_{\delta h(n)}^{\rho_n-\alpha} \frac{d\lambda}{\phi^{(n)}(\lambda)}$ 
tends to $\int_0^{\rho-\alpha}\frac{d\lambda}{\phi(\lambda)}$,
since $h\in\mathcal{Z}_0$. On the other hand, from Fatou's lemma, 
\[\int_0^{\rho-\alpha}\frac{d\lambda}{\phi(\lambda)}\le\liminf_{n\rightarrow\infty}
\int_0^{\rho_n-\alpha} \frac{d\lambda}{h(n)+\phi^{(n)}(\lambda)}\,.\]
Since $\delta$ is chosen arbitrarily, we derive that
\begin{equation}\label{4267}
\lim_{n\rightarrow\infty}
\int_0^{\rho_n-\alpha} \frac{d\lambda}{h(n)+\phi^{(n)}(\lambda)}=
\int_0^{\rho-\alpha}\frac{d\lambda}{\phi(\lambda)}<\infty.
\end{equation}
Then (\ref{4267}) together with (\ref{9910}) show that 
$\int_{0}^{\rho_n}\frac{ d\lambda}{h(n)+\phi^{(n)}(\lambda)}=O(\log(h(n)))$.
Moreover, since $\psi^{(n)}(0)=\rho_n$ and $\psi^{(n)'}(h(n))\rightarrow\psi^{'}(0)\in(0,\infty)$,
$\psi^{(n)}(h(n))\sim \rho+h(n)\psi^{'}(0)$, as $n\uparrow\infty$ and therefore, 
$e^{-\rho_n x}-e^{-\psi^{(n)}(h(n))x}\sim \psi^{'}(0)h(n)$ as $n\rightarrow\infty$. Then 
we can conclude that the second term of the right hand side of (\ref{7753}) satisfies,
$\lim_{n\rightarrow\infty}
(e^{-\rho_n x}-e^{-\psi^{(n)}(h(n))x})\int_0^{\rho_n}\frac{ d\lambda}{h(n)+\phi^{(n)}(\lambda)}=0$.

Let us now write the first term of the right hand side of (\ref{7753}) as
\[\int_{0}^{\rho_n}\frac{e^{-\lambda x}-e^{-\rho_n x}}{h(n)+\phi^{(n)}(\lambda)}d\lambda=
\int_{0}^{\rho_n-\alpha}\frac{e^{-\lambda x}-e^{-\rho_n x}}{h(n)+\phi^{(n)}(\lambda)}d\lambda+
\int_{\rho_n-\alpha}^{\rho_n}\frac{e^{-\lambda x}-e^{-\rho_n x}}{h(n)+\phi^{(n)}(\lambda)}d\lambda,\]
where $\alpha$ is as above and $n$ is sufficiently large so that $\rho_n-\alpha>0$. Then we prove that 
\[\lim_{n\rightarrow\infty}
\int_{0}^{\rho_n-\alpha}\frac{e^{-\lambda x}-e^{-\rho_n x}}{h(n)+\phi^{(n)}(\lambda)}d\lambda
=\int_{0}^{\rho-\alpha}\frac{e^{-\lambda x}-e^{-\rho x}}{\phi(\lambda)}d\lambda\]
in the same way as for (\ref{4267}). For the second term, by using (\ref{7501}), we obtain that 
$\frac{e^{-\lambda x}-e^{-\rho_n x}}{h(n)+\phi^{(n)}(\lambda)}\ind_{[\rho_n-\alpha,\rho_n]}(\lambda)$ 
is bounded by a constant and we derive from dominated convergence that 
\[\lim_{n\rightarrow\infty}
\int_{\rho_n-\alpha}^{\rho_n}\frac{e^{-\lambda x}-e^{-\rho_n x}}{h(n)+\phi^{(n)}(\lambda)}d\lambda
=\int_{\rho-\alpha}^{\rho}\frac{e^{-\lambda x}-e^{-\rho x}}{\phi(\lambda)}d\lambda.\]
Then we have proved that 
$\lim_{n\rightarrow\infty}\int_0^{\rho_n} \frac{e^{-\lambda x}-e^{-\psi^{(n)}(h(n))x}}
{h(n)+\phi^{(n)}(\lambda)}d\lambda=\int_0^{\rho}\frac{e^{-\lambda x}-e^{-\rho x}}{\phi(\lambda)}d\lambda$.
The fact that $\int_0^{\rho}\frac{e^{-\lambda x}-e^{-\rho x}}{\phi(\lambda)}d\lambda<\infty$
is a straightforward consequence of the behaviour of $\phi$ around $\rho$, that is, 
$1/\phi(\lambda)=O(1/(\rho-\lambda))$, as $\lambda\uparrow\rho$. 
\end{proof}

\begin{corollary}\label{1153} Let ${\rm e}_n$ be as in Lemma $\ref{3722}$ and recall that 
$h\in\mathcal{Z}_0$. Then for all $x>0$, 
\[\int_0^{{\rm e}_n}\frac{du}{X^{(n)}_u}\ind_{\{\tau^{(n)}=\infty\}}
\xrightarrow[n\to\infty]{\mathrm{L}^1(\p_x)}\zeta\ind_{\{\zeta<\infty\}}\,.\]
\end{corollary}
\begin{proof} Let us first observe that from the representation (\ref{2675}),
\begin{equation}\label{4200}
\zeta\ind_{\{\zeta<\infty\}}=\int_0^{\infty}\frac{du}{X_u}\ind_{\{\tau=\infty\}}\,,
\end{equation}
and write, 
\begin{eqnarray*}
&&\left|\int_0^{{\rm e}_n}\frac{du}{X^{(n)}_u}\ind_{\{\tau^{(n)}=\infty\}}-\zeta\ind_{\{\zeta<\infty\}}
\right|\le\left|\int_0^{{\rm e}_n}\frac{du}{X^{(n)}_u}\ind_{\{\tau^{(n)}=\infty\}}-
\int_0^{{\rm e}_n}\frac{du}{X_u}\ind_{\{\tau=\infty\}}
\right|+\int_{{\rm e}_n}^{\infty}\frac{du}{X_u}\ind_{\{\tau=\infty\}}\\
&&\le\left(\int_0^{{\rm e}_n}\frac{du}{X^{(n)}_u}
\ind_{\{\tau^{(n)}=\infty\}}-\int_0^{{\rm e}_n}\frac{du}{X_u}\ind_{\{\tau^{(n)}=\infty\}}\right)
+\left(\int_0^{{\rm e}_n}\frac{du}{X_u}
\ind_{\{\tau=\infty\}}-\int_0^{{\rm e}_n}\frac{du}{X_u}\ind_{\{\tau^{(n)}=\infty\}}\right)\\
&&+\int_{{\rm e}_n}^{\infty}\frac{du}{X_u}\ind_{\{\tau=\infty\}}.
\end{eqnarray*}
From dominated convergence and Lemma \ref{3872}, the last term of the right hand side of this 
inequality tends to 0 in expectation. (Note that since $X^{(n)}\le X$ both terms between parentheses  
are non negative.) Then the expectation of each of the four terms between the parenthesis tends to 
$\e_x\left(\int_0^{\infty}\frac{du}{X_u}\ind_{\{\tau=\infty\}}\right)$, as $n\rightarrow\infty$.
For the first and the third term, this is a direct consequence of Lemma \ref{3872}.
It only remains to check that 
$\lim_{n\rightarrow\infty}\e_x\left(\int_0^{{\rm e}_n}\frac{du}{X_u}\ind_{\{\tau^{(n)}=\infty\}}\right)=
\e_x\left(\int_0^{\infty}\frac{du}{X_u}\ind_{\{\tau=\infty\}}\right)$.
But again it follows from Lemma \ref{3872} and dominated convergence.
\end{proof}

Let us recall some facts on scale functions that will be useful for the proof of the next lemma. We refer to 
\cite{kkr} for more details. The scale function $W^{(n)}$ of $X^{(n)}$ is a continuous increasing function 
whose Laplace transform is given by
\begin{equation}\label{6726}
\int_0^\infty e^{-\lambda x}W^{(n)}(x)\,dx=\frac{1}{|\phi^{(n)}|(\lambda)}\,,\;\;\lambda\ge\rho_n\,.
\end{equation}
Recall also that this function admits the following representation,
\begin{equation}\label{6725}
W^{(n)}(x)=e^{\rho_n x}\int_0^xe^{-\rho_n u}\,U_n(du)\,,\;\;x\ge0\,,
\end{equation}
where $U_n$ is the potential measure of the upward ladder height process $H^{(n)}$ of $X^{(n)}$, see (40) in 
\cite{kkr}. 
Note also that we from a general result for subordinators, see Proposition III.1 in 
\cite{be} and inequality (5) in its proof , there is a universal constant $C$ such that,
\begin{equation}\label{6724}
U_n([0,x])\le C\frac{|1/x-\rho_n|}{|\phi^{(n)}|(1/x)}\,,\;\;x>0\,.
\end{equation}
In particular, $C$ depends neither on $x$ nor on $n$.

\begin{lemma}\label{5622}
There is a constant $C\in(0,\infty)$ such that for all $x\ge0$ 
and $n\ge1$, 
\[\e_x\left(\int^{\infty}_{0}\frac{du}{X^{(n)}_u}
\ind_{\left\{\tau^{(n)}=\infty,\,X^{(n)}_u<1/h(n)\right\}}\right)\le C\,.\]
If $(x_n)$ is any sequence such that $x_n\ge 1/h(n)$ for all $n$, then
\[\lim_{n\rightarrow\infty}\e_{x_n}\left(\int^{\infty}_{0}\frac{du}{X^{(n)}_u}
\ind_{\left\{\tau^{(n)}=\infty,\,X^{(n)}_u<1/h(n)\right\}}\right)=0\,.\]
\end{lemma}
\begin{proof} Dividing the expression (\ref{2506}) for $X^{(n)}$ by $\varepsilon$ and letting $\varepsilon$ 
going to 0, gives, for all $\lambda\ge\rho_n$,
\begin{eqnarray}
&&\e_x\left(\int_0^\infty e^{-\lambda X_u^{(n)}}\ind_{\{\tau^{(n)}=\infty\}}\,du\right)=
\frac{e^{-\rho_nx}-e^{-\lambda x}}{\varphi^{(n)}(\lambda)}-\frac{e^{-\rho_nx}-e^{-(\lambda+\rho_n)x}}
{\varphi^{(n)}(\lambda+\rho_n)}\nonumber\\
&=&\int_0^{\infty}e^{-\lambda y}(1-e^{-\rho_ny})
\left(e^{-\rho_nx}W^{(n)}(y)-W^{(n)}(y-x)\ind_{\{y\ge x\}}\right)\,dy,\label{8111}
\end{eqnarray}
where the second equality follows from (\ref{6726}) (for $X^{(n)}$). The set of functions $y\mapsto e^{-\lambda y}$,
$\lambda\ge\rho_n$, for any fixed $n\ge0$, is total in the vector space of continuous functions on $(0,\infty)$.
Therefore, through classical arguments, we can extend the last identity to any non negative, measurable function
defined on $(0,\infty)$. In particular, replacing $y\mapsto e^{-\lambda y}$ by 
$\displaystyle y\mapsto \frac1y\ind_{\{y<1/h(n)\}}$ gives,
\begin{eqnarray}
&&\e_x\left(\int^{\infty}_{0}\frac{du}{X^{(n)}_u}
\ind_{\left\{\tau^{(n)}=\infty,\,X^{(n)}_u<1/h(n)\right\}}\right)\nonumber\\
&=&\int_0^{1/h(n)}\frac{1-e^{-\rho_ny}}{y}
\left(e^{-\rho_nx}W^{(n)}(y)-W^{(n)}(y-x)\ind_{\{y\ge x\}}\right)\,dy\label{2688}\\
&=&\int_0^{x}\frac{1-e^{-\rho_ny}}{y}e^{-\rho_nx}W^{(n)}(y)\,dy+
\int_{x}^{1/h(n)}\frac{1-e^{-\rho_ny}}{y}\left(e^{-\rho_nx}W^{(n)}(y)-W^{(n)}(y-x)\right)\,dy\,,\nonumber
\end{eqnarray}
where we assume first that $x\le 1/h(n)$. Note that the sequence 
$(\phi^{(n)'}(\rho_n))=(\phi'(\rho_n+\varepsilon_n))$ tends to $\phi'(\rho)<0$. Hence it is bounded away 
from 0 so that from (53) in the proof of Lemma 3.3 in \cite{kkr}, there is a constant 
$C$ that depends neither on $x$ nor on $n$, such that $W^{(n)}(x)\le Ce^{\rho_n x}$, for all $x\ge0$. Since 
$(\rho_n)$ is increasing and tends to $\rho$, we obtain the bound,
\begin{equation}\label{0354}
\int_0^{x}\frac{1-e^{-\rho_ny}}{y}e^{-\rho_nx}W^{(n)}(y)\,dy\le C\int_0^{x}
\frac{1-e^{-\rho y}}{y}e^{-\rho_0(x-y)}\,dy\,,
\end{equation}
and we readily show that the right hand side of this inequality tends to 0 as $x$ tends to $\infty$. 

Now let us consider the second term of the last member of (\ref{2688}) and apply (\ref{6725}) in order to obtain, 
\begin{eqnarray}
&&\int_{x}^{1/h(n)}\frac{1-e^{-\rho_ny}}{y}\left(e^{-\rho_nx}W^{(n)}(y)-W^{(n)}(y-x)\right)\,dy\nonumber\\
&&=\int_{x}^{1/h(n)}\frac{1-e^{-\rho_ny}}{y}e^{-\rho_n x}\left(\int_{y-x}^{y}e^{-\rho_n(u-y)}U_n(du)\right)\,dy\,.
\label{1064}
\end{eqnarray}
Assume first that $x$ is less than some constant $c>0$, that is $x\le c$, and write
\begin{eqnarray}
&&\int_{x}^{1/h(n)}\frac{1-e^{-\rho_ny}}{y}e^{-\rho_n x}\left(\int_{y-x}^{y}e^{-\rho_n(u-y)}U_n(du)\right)\,dy
\le\int_{x}^{1/h(n)}\frac{1-e^{-\rho_ny}}{y}\int_{y-x}^{y}U_n(du)\,dy\nonumber\\
&\le&\int_{0}^{1/h(n)}\int_{u}^{u+x}\frac{1-e^{-\rho y}}{y}\,dy\,U_n(du)\nonumber\\
&=&\int_{0}^{1}\int_{u}^{u+x}\frac{1-e^{-\rho y}}{y}\,dy\,U_n(du)+\int_{1}^{1/h(n)}
\int_{u}^{u+x}\frac{1-e^{-\rho y}}{y}\,dy\,U_n(du)\,.\label{2477}
\end{eqnarray}
The function $(x,u)\mapsto\int_{u}^{u+x}\frac{1-e^{-\rho y}}{y}\,dy$ is continuous on $[0,c]\times[0,1]$ and
from the convergence of $X^{(n)}$ toward $X$ we derive that the measure $U_n(du)$ converges weakly toward $U(du)$.
Hence the first term of (\ref{2477}) is bounded uniformly in $x\in[0,c]$ and $n\ge0$. 
On the other hand, an integration by part gives for the second term,
\begin{eqnarray*}
&&\int_{1}^{1/h(n)}\int_{u}^{u+x}\frac{1-e^{-\rho y}}{y}\,dy\,U_n(du)\le\int_1^{1/h(n)}\frac cu U_n(du)\\
&&=\int_1^{1/h(n)}\frac c{u^2}U_n([0,u])\,du+c\frac{U_n([0,1/h(n)])}{1/h(n)}-cU_n([0,1])\,,
\end{eqnarray*}
and (\ref{6724}) yields after a change of variable, 
\begin{equation}\label{1006}
\int_1^{1/h(n)}\frac cuU_n(du)\le C\left(\int_{h(n)}^{1}\frac{v-\rho_n}{\varphi^{(n)}(v)}\,dv+
\left(h(n)-\rho_n\right)\frac{h(n)}{\varphi^{(n)}(h(n))}\right)\,.
\end{equation}
Our assumption on $h$ (i.e.~$h\in\mathcal{Z}_0$) and the fact that 
$\lim_{n\rightarrow\infty}\frac{h(n)}{\varphi^{(n)}(h(n))}=0$ shows that the term above is also uniformly 
bounded  in $x\in[0,c]$ and $n\ge0$. 

Assume now that $x\ge c$ and note that the right hand side of (\ref{1064}) is less than
\begin{eqnarray*}
\int_{x}^{1/h(n)}\frac{e^{-\rho_n x}}y\left(\int_{y-x}^{y}e^{-\rho_n(u-y)}U_n(du)\right)\,dy
=e^{-\rho_n x}\int_{0}^{1/h(n)}\int_{u\vee x}^{(u+x)\wedge 1/h(n)}\frac{e^{\rho_n(y-u)}}y\,dy\,U_n(du)\,.
\end{eqnarray*}
Then on the one hand, 
\begin{eqnarray*}
e^{-\rho_n x}\int_{0}^{c/2}\int_{u\vee x}^{(u+x)\wedge 1/h(n)}\frac{e^{\rho_n(y-u)}}y\,dy\,U_n(du)&=&
e^{-\rho_n x}\int_{0}^{c/2}\int_{u\vee x-u}^{(u+x)\wedge 1/h(n)-u}\frac{e^{\rho_n v}}{u+v}\,dv\,U_n(du)\\
&\le&e^{-\rho_n x}\int_{0}^{c/2}\int_{x-u}^{x}\frac{e^{\rho_n v}}{u+v}\,dv\,U_n(du)\\
&\le&\int_{0}^{c/2}\int_{x-c/2}^{x}\frac{1}{u+v}\,dv\,U_n(du)\\
&\le&\frac c2\int_{0}^{c/2}\frac{1}{u+x-c/2}\,U_n(du)\le U_n([0,c/2])\,.
\end{eqnarray*}
Since $U_n([0,c/2])$ converges toward $U([0,c/2])$ this last term is uniformly bounded in $x\ge c$ and $n\ge0$.
On the other hand, 
\begin{eqnarray*}
&&e^{-\rho_n x}\int_{c/2}^{1/h(n)}\int_{u\vee x}^{(u+x)\wedge 1/h(n)}\frac{e^{\rho_n(y-u)}}y\,dy\,U_n(du)=
e^{-\rho_n x}\int_{c/2}^{1/h(n)}\int_{u\vee x-u}^{(u+x)\wedge 1/h(n)-u}\frac{e^{\rho_n v}}{u+v}\,dv\,U_n(du)\\
&&\le e^{-\rho_n x}\int_{c/2}^{1/h(n)}\int_{0}^{x}\frac{e^{\rho_n v}}{u+v}\,dv\,U_n(du)\le
e^{-\rho_n x}\int_{c/2}^{1/h(n)}\frac{e^{\rho_n x}-1}{\rho_n u}\,U_n(du)\le
\int_{c/2}^{1/h(n)}\frac{1}{\rho_n u}\,U_n(du)\,.
\end{eqnarray*}
But this last term can be bounded uniformly in $x\ge c$ and $n\ge0$ exactly as in (\ref{1006}).
Then we proved that the term in (\ref{2688}) is bounded by some constant for all $n$ and $x$ such that $x\le 1/h(n)$.

Now for $n$ and $x$ such that $x\ge 1/h(n)$, the term in (\ref{2688}) becomes 
\[\int_0^{x}\frac{1-e^{-\rho_ny}}{y}e^{-\rho_nx}W^{(n)}(y)\,dy\,.\]
But then the bound obtained in (\ref{0354}) is still valid and gives the result in this case. 
The second assertion of the lemma also follows from this argument. 
\end{proof}

We define the first passage time above the level $x\ge0$ of $X^{(n)}$ by
\[\tau^{(n)}_{x}=\inf\{t:X^{(n)}_t\ge x\}\,.\]

\begin{lemma}\label{5322}
Let ${\rm e}_n$ be as in Lemma $\ref{3722}$. Then, for all $x>0$,  
\[\lim_{n\rightarrow\infty}\int_{\tau^{(n)}_{1/h(n)}}^{{\rm e}_n}\frac{du}{X^{(n)}_u}
\ind_{\{\tau^{(n)}=\infty\}}\xrightarrow[n\to\infty]{\mathrm{L}^1(\p_x)}0.\]
\end{lemma}
\begin{proof} Let us first show that 
\begin{equation}\label{8477}\lim_{n\rightarrow\infty}\int_{\tau^{(n)}_{1/h(n)}}^{{\rm e}_n}\frac{du}{X^{(n)}_u}
\ind_{\left\{\tau^{(n)}=\infty,\,\tau^{(n)}_{1/h(n)}<{\rm e}_n\right\}}\xrightarrow[n\to\infty]{\mathrm{L}^1(\p_x)}0.
\end{equation}
Note that 
\[\int_{\tau^{(n)}_{1/h(n)}}^{{\rm e}_n}\frac{du}{X^{(n)}_u}
\ind_{\left\{\tau^{(n)}=\infty,\,\tau^{(n)}_{1/h(n)}<{\rm e}_n\right\}}\le
\int_{\tau^{(n)}_{1/h(n)}}^{\tau^{(n)}_{1/h(n)}+{\rm e}_n}\frac{du}{X^{(n)}_u}
\ind_{\left\{\tau^{(n)}=\infty\right\}}\,.\]
Therefore, from the strong Markov property, the inequality $1/h(n)\le X^{(n)}(\tau^{(n)}_{1/h(n)})$, $\p_x$-a.s. 
and Lemma \ref{3722}, 
\begin{eqnarray*}
\e_x\left(\int_{\tau^{(n)}_{1/h(n)}}^{{\rm e}_n}\frac{du}{X^{(n)}_u}
\ind_{\left\{\tau^{(n)}=\infty,\,\tau^{(n)}_{1/h(n)}<{\rm e}_n\right\}}\right)&\le&
E_{1/h(n)}\left(\int_{0}^{{\rm e}_n}\frac{du}{X^{(n)}_u}\ind_{\left\{\tau^{(n)}=\infty\right\}}\right)\\
&=&\int_0^{\rho_n} \frac{e^{-\lambda n}-e^{-\psi^{(n)}(h(n)) n}}{h(n)+\phi^{(n)}(\lambda)} 
d\lambda
\end{eqnarray*}
and this last term tends to 0 from Lemma \ref{3722} and dominated convergence. Hence (\ref{8477}) is proved.  

It remains to show that 
\begin{equation}\label{8447}
\lim_{n\rightarrow\infty}\int^{\tau^{(n)}_{1/h(n)}}_{{\rm e}_n}\frac{du}{X^{(n)}_u}
\ind_{\left\{\tau^{(n)}=\infty,\,\tau^{(n)}_{1/h(n)}>{\rm e}_n\right\}}\xrightarrow[n\to\infty]{\mathrm{L}^1(\p_x)}0.
\end{equation}
Let us note that 
\[\e_x\left(\int^{\tau^{(n)}_{1/h(n)}}_{{\rm e}_n}\frac{du}{X^{(n)}_u}
\ind_{\left\{\tau^{(n)}=\infty,\,\tau^{(n)}_{1/h(n)}>{\rm e}_n\right\}}\right)\le
\e_x\left(\int^{\infty}_{{\rm e}_n}\frac{du}{X^{(n)}_u}
\ind_{\left\{\tau^{(n)}=\infty,\,X^{(n)}_u<1/h(n)\right\}}\right)\,,\]
and
\begin{eqnarray*}
&&\e_x\left(\int^{\infty}_{{\rm e}_n}\frac{du}{X^{(n)}_u}
\ind_{\left\{\tau^{(n)}=\infty,\,X^{(n)}_u<1/h(n)\right\}}\right)=
\e_x\left(\ind_{\{\underline{X}_{e_n}^{(n)}\ge0\}}\left(\int^{\infty}_{0}\frac{du}{X^{(n)}_u}
\ind_{\{\tau^{(n)}=\infty,\,X^{(n)}_u<1/h(n)\}}\right)\circ\theta_{e_n}\right)\\
&&=\e_x\left(\ind_{\{\underline{X}_{e_n}^{(n)}\ge0\}}\e_{X_{e_n}^{(n)}}\left(\int^{\infty}_{0}\frac{du}{X^{(n)}_u}
\ind_{\{\tau^{(n)}=\infty,\,X^{(n)}_u<1/h(n)\}}\right)\right)\,.
\end{eqnarray*}
From the law of large numbers for L\'evy processes, for all $m\ge0$, 
$\lim_{n\rightarrow\infty}h(n)X^{(m)}_{e_n}=E(X_1^{(m)})=\phi^{(m)'}(0)\le
\lim_{n\rightarrow\infty}h(n)X_{e_n}^{(n)}$, a.s. Since 
$\lim_{m\rightarrow\infty}\phi^{(m)'}(0)=\infty$, we obtain that
$\lim_{n\rightarrow\infty}h(n)X_{e_n}^{(n)}=\infty$, a.s. Then (\ref{8447}) follows from Lemma \ref{5622}
and dominated convergence. 
\end{proof}

We recall the definition of the first passage time above the level $x\ge0$ of $Z^{(n)}$:
\[\sigma^{(n)}_{x}=\inf\{t:Z^{(n)}_t\ge x\}\,.\]

\noindent {\it Proof of Theorem $\ref{0486}$}. Thanks to Corollary \ref{1153} and Lemma \ref{5622}, and using
the decomposition,
\begin{equation*}
    \sigma^{(n)}_{1/h(n)}\1{\tau^{(n)}=\infty} = \int_0^{{\rm e}_n} \frac{du}{X_u^{(n)}}\1{\tau^{(n)}=\infty} - 
    \int^{{\rm e}_n}_{\tau_{1/h(n)}^{(n)}} \frac{du}{X_u^{(n)}}\1{\tau^{(n)}=\infty}\,,
\end{equation*}
we obtain $\displaystyle \sigma^{(n)}_{1/h(n)}\1{\tau^{(n)}=\infty}\xrightarrow[]{L_1}\zeta\1{\zeta<\infty}$. 
Now observe that $\sigma^{(n)}_{1/h(n)}\1{\tau^{(n)}=\infty}\le 
\sigma^{(n)}_{1/h(n)}\1{\sigma^{(n)}_{1/h(n)}<\infty}$, $\p_x$-a.s. Then we shall conclude by proving that,
\begin{equation}\label{1530}
    \lim_{n\to\infty}\e_x\left(\sigma^{(n)}_{1/h(n)}\1{\sigma^{(n)}_{1/h(n)}<\infty} \right) = 
    \lim_{n\to\infty}\e_x\left(\sigma^{(n)}_{1/h(n)}\1{\tau^{(n)}=\infty} \right)\,.
\end{equation}
On the one hand, an application of the Markov property allows us 
to define the process $X^{(n)}$ conditioned to stay positive as follows: for all $t\geq 0$ 
and $\Lambda\in\mathcal{F}_t^{(n)}$, 
\begin{equation*}
\e_x(\ind_\Lambda\,|\,\tau^{(n)}=\infty)=
\e_x\left(\frac{1-e^{-\rho_n X_t^{(n)}}}{1-e^{-\rho_n x}}\ind_\Lambda\ind_{\{t<\tau^{(n)}\}}\right)\,,
\end{equation*}
where $(\mathcal{F}_t^{(n)})$ is the natural filtration generated by the process $X^{(n)}$.
Extending this formula from time $t$ to the stopping time $\tau_{1/h(n)}^{(n)}$ and replacing 
$\ind_\Lambda$ by the $\mathcal{F}(\tau_{1/h(n)}^{(n)})$-measurable functional 
$\int_0^{\tau_{1/h(n)}^{(n)}}du/X^{(n)}_u = \sigma^{(n)}_{1/h(n)}$ in the above 
formula allows us to write, 
\begin{equation*} 
\e_x\left(\left.\sigma^{(n)}_{1/h(n)}\,\right|\,\tau^{(n)}=\infty\right)=
\e_x\left(\frac{1-e^{-\rho_n X^{(n)}(\tau_{1/h(n)}^{(n)})}}{1-e^{-\rho_n x}}
\sigma^{(n)}_{1/h(n)}\ind_{\{\tau_{1/h(n)}^{(n)}<\tau^{(n)}\}}\right)\,.
\end{equation*}
Then from this relation together with the equality 
$\{\sigma^{(n)}_{1/h(n)}<\infty\}=\{\tau^{(n)}_{1/h(n)}<\tau^{(n)}\}$ and the inequality
$X^{(n)}_{\tau_{1/h(n)}^{(n)}}\geq 1/h(n)$, we obtain,
\begin{eqnarray}
\e_x\left(\sigma_{1/h(n)}^{(n)}\ind_{\{\sigma_{1/h(n)}^{(n)}<\infty\}}\right)&=&
\e_x\left(\sigma_{1/h(n)}^{(n)}\ind_{\{\tau_{1/h(n)}^{(n)}<\tau^{(n)}\}}\right)\nonumber\\
&\leq& \e_x\left( \frac{1-e^{-\rho_n x}}{1-e^{-\rho_n/h(n)}} 
\frac{1-e^{-\rho_n X^{(n)}(\tau_{1/h(n)}^{(n)})}}{1-e^{-\rho_n x}}
\sigma_{1/h(n)}^{(n)}\ind_{\{\tau_{1/h(n)}^{(n)}<\tau^{(n)}\}}\right)\nonumber\\
&=&\frac{1}{1-e^{-\rho_n/h(n)}}\e_x\left(\sigma_{1/h(n)}^{(n)}\,\1{\tau^{(n)}=\infty}\right)\,.\label{1371}
\end{eqnarray}
On the other hand, the relation 
\[\{ \sigma^{(n)}_{1/h(n)}<\infty\} = \{ \tau^{(n)}_{1/h(n)}<\tau^{(n)}\}= \{\tau^{(n)}=\infty \}
\cup \{\tau^{(n)}_{1/h(n)}<\tau^{(n)},\tau^{(n)}<\infty\}\,,\]
yields,
\begin{equation*}
    \e_x\left(\sigma_{1/h(n)}^{(n)}\1{\tau^{(n)}=\infty}\right) \leq 
    \e_x\left(\sigma_{1/h(n)}^{(n)}\ind_{\{\sigma_{1/h(n)}^{(n)}<\infty\}}\right)\,.
\end{equation*}
This inequality together with (\ref{1371}) allow us to obtain (\ref{1530}) and concludes the proof.
$\hfill\Box$\\

\subsubsection{Almost sure convergence}

\begin{lemma}\label{67}
    For all $x>0$, there exists a constant $C$ such that for all $n\ge0$, 
    \begin{equation*}
        \e_x\left(\tau_{1/h(n)}^{(n)}\1{\tau^{(n)}=\infty}\right)\leq \frac{C}{\phi^{(n)}(h(n))}\, .
    \end{equation*}
\end{lemma}
\begin{proof}
If $X$ is a subordinator, then $\e_x(\tau^{(n)}_{1/h(n)}\1{\tau^{(n)}=\infty})=\e_x(\tau^{(n)}_{1/h(n)})$, 
and the result follows from Proposition \rm{III}.1 in \cite{be}. If $X$ is not a subordinator, then write
\begin{equation*}
    \e_x\left(\tau^{(n)}_{1/h(n)}\1{\tau^{(n)}=\infty}\right) \le 
    \e_x\left( \int_0^\infty \1{X_t^{(n)}\leq1/h(n)}\1{\tau^{(n)}=\infty}dt\right)\,.
\end{equation*}
As already justified in the beginning of the proof of Lemma \ref{5622}, formula (\ref{8111}) may be extended 
to the function $y\mapsto \1{y\leq1/h(n)}$, so that 
\begin{equation*}
    \e_x\left( \int_0^\infty \1{X_t^{(n)}\leq1/h(n)}\1{\tau^{(n)}=\infty}\right) = \int_0^{1/h(n)}(1-e^{-\rho_ny})
\left(e^{-\rho_nx}W^{(n)}(y)-W^{(n)}(y-x)\ind_{\{y\ge x\}}\right)dy.
\end{equation*}
With calculations similar to (\ref{0354}) and (\ref{1064}) , we find that this last term is less or equal than
\begin{align*}
    c_x + \int_{0}^{1/h(n)}\int_{u}^{u+x}(1-e^{-\rho y})\,dy\,U_n(du)\, &\leq c_x + xU_n([0,1/h(n)]) \\
    &\leq c_x + c_1x\frac{h(n)-\rho_n}{\varphi^{(n)}(h(n))}\leq\frac{C}{\phi^{(n)}(h(n))}\,,
\end{align*}
where $c_x$ and $c_1$ are constants which do not depend on $n$ and where we used (\ref{6724}) 
for the last inequality.
\end{proof}

\begin{lemma}\label{447} 
Let ${\rm e}_n$ be as in Lemma $\ref{3722}$.
\begin{itemize}
\item[$1.$]
If $\displaystyle \sum_{n\geq 0} \frac{\phi(\varepsilon_n)}{\phi(\varepsilon_n+h(n))}<\infty$, then 
$\left( (X_{\tau^{(n)}_{1/h(n)}}-X^{(n)}_{\tau^{(n)}_{1/h(n)}})\1{\tau_n=\infty}\right)_{n\ge0}$ converges 
toward $0$, $\p_x$-almost surely, for all $x>0$.
\item[$2.$]
 If $\displaystyle \sum_{n\geq 0} \frac{\phi(\varepsilon_n)}{h(n)}<\infty$, then 
 $\left( (X_{{\rm e}_n}-X^{(n)}_{{\rm e}_n})\1{\tau_n=\infty}\right)_{n\ge0}$ converges toward $0$, 
 $\p_x$-almost surely, for all $x>0$.
\end{itemize}
\end{lemma}
\begin{proof} Let ${\rm e}$ be an exponentially distributed random variable with parameter 1 which is 
independent of the processes $X$ and $X^{(n)}$, $n\ge0$. To simplify our calculations, let us denote by 
$(Y^{(n)}_t,t\ge 0)$ the process $(X_t - X^{(n)}_t, t\geq 0)$ and recall that from Theorem \ref{2065}, 
this process  is a subordinator with Laplace exponent $\varphi-\varphi^{(n)}$ and that it is independent 
of $X^{(n)}$. Then for all $a>0$ and $n\ge0$,  
\begin{eqnarray}\label{4640}
&&\p_x\left(Y^{(n)}_{\tau^{(n)}_{1/h(n)}}\1{\tau^{(n)}=\infty} > a {\rm e}\right)=
\int_0^\infty e^{-t} \p_x\left(Y^{(n)}_{\tau^{(n)}_{1/h(n)}}\1{\tau^{(n)}=\infty}> a t\right)dt\\
&&\geq \int_0^1 e^{-t} \p_x\left(Y^{(n)}_{\tau^{(n)}_{1/h(n)}}\1{\tau^{(n)}=\infty}> a\right) dt \nonumber
\geq(1-e^{-1}) \p_x\left(Y^{(n)}_{\tau^{(n)}_{1/h(n)}}\1{\tau^{(n)}=\infty}> a\right).
\end{eqnarray}
On the other hand conditioning by ${\rm e}$ and then by $(\tau^{(n)}_{1/h(n)},\tau^{(n)})$ yields,
\begin{align*}
\p_x\left(Y^{(n)}_{\tau^{(n)}_{1/h(n)}}\1{\tau^{(n)}=\infty}> a {\rm e}\right) &= 1 - 
\e_x\left( \exp\left(-\frac1aY^{(n)}_{\tau^{(n)}_{1/h(n)}}\1{\tau^{(n)}=\infty}\right)\right) \\
&= 1 - \e_x\left( \exp\left(-\frac{1}{a}Y^{(n)}_{\tau^{(n)}_{1/h(n)}}\right)\1{\tau^{(n)}=\infty} + 
\1{\tau^{(n)}<\infty} \right) \\
&=\e_x\left(\left(1-e^{-(\varphi(1/a)-\varphi^{(n)}(1/a))\tau^{(n)}_{1/h(n)}}\right)\1{\tau^{(n)}=
\infty}\right)
\end{align*}
and this last quantity is less or equal than 
$(\varphi(1/a)-\varphi^{(n)}(1/a))\e_x(\tau^{(n)}_{1/h(n)}\1{\tau^{(n)}=\infty})$. 
Now note that $\displaystyle \sum_{n\geq 0}\frac{\phi(\varepsilon_n)}{\phi^{(n)}(h(n))}<\infty$ 
if and only if $\displaystyle \sum_{n\geq 0}\frac{\phi(\varepsilon_n)}{\phi(\varepsilon_n+h(n))}<\infty$.
Then using the series expansion $\varphi(1/a)-\varphi^{(n)}(1/a) = \varphi(\varepsilon_n) + 
\varepsilon_n \varphi'(1/a) + o(\varepsilon_n)$, the fact that $\varepsilon_n = o(\varphi(\varepsilon_n))$, 
Lemma \ref{67} and the hypothesis give $\displaystyle\sum_{n\geq 1} 
\p_x\left(Y^{(n)}_{\tau^{(n)}_{1/h(n)}}\1{\tau^{(n)}=\infty}>a{\rm e}\right)<\infty $. The result follows 
from (\ref{4640}) and Borel Cantelli's lemma.

The proof of the almost sure convergence of the sequence 
$\left( (X_{{\rm e}_n}-X^{(n)}_{{\rm e}_n})\1{\tau_n=\infty}\right)_{n\ge0}$ toward 0 is very similar, so
it is omitted.
\end{proof}

\noindent {\it Proof of Theorem $\ref{1086}$}.
Let $n\ge0$, recall the expression (\ref{4200}) and write 
$\sigma_{1/h(n)}^{(n)} \1{\tau^{(n)}=\infty} - \zeta\1{\zeta<\infty}$ as
\begin{align*}
    \int_0^{\tau^{(n)}_{1/h(n)}} \frac{du}{X^{(n)}_u}\1{\tau^{(n)}= \infty} - \int_0^{\tau^{(n)}_{1/h(n)}} 
    \frac{du}{X_u}\1{\tau^{(n)}= \infty} 
     &+ \int_0^{\tau^{(n)}_{1/h(n)}} \frac{du}{X_u}(\1{\tau^{(n)}= \infty}-\1{\tau = \infty})   \\
    &-\int_{\tau^{(n)}_{1/h(n)}}^\infty \frac{du}{X_u}\1{\tau = \infty}\,.
\end{align*}
The last two terms tend $\p_x$-a.s to $0$ as $n$ goes to $\infty$, from dominated convergence.
For the two first terms, recall that the process $(X_t - X^{(n)}_t, t\geq 0)$ is a subordinator, so that
\begin{align*}
     \int_0^{\tau^{(n)}_{1/h(n)}} \left(\frac{1}{X^{(n)}_u} -  \frac{1}{X_u}\right)\1{\tau^{(n)}=\infty}du 
     \leq \frac{X_{\tau^{(n)}_{1/h(n)}}-X_{\tau^{(n)}_{1/h(n)}}^{(n)}}{\inf_{u\in[0,\tau^{(n)}_{1/h(n)}]}X^{(n)}_u} 
     \int_0^{\tau^{(n)}_{1/h(n)}} \frac{du}{X_u}\1{\tau^{(n)}=\infty}du.
\end{align*}
Then the right hand side of this inequality tends $\p_x$-a.s.~to 0 since on the one hand, $\p_x$-a.s.,       
\begin{eqnarray*}
\lim_{n\rightarrow\infty}\int_0^{\tau^{(n)}_{1/h(n)}}\frac{du}{X_u}\1{\tau^{(n)}=\infty}du&=&
\int_0^{\infty}\frac{du}{X_u}\1{\tau=\infty}du<\infty\,,\\
\lim_{n\rightarrow\infty}\frac{\1{\tau^{(n)}=\infty}}{\inf_{u\in[0,\tau^{(n)}_{1/h(n)}]}X^{(n)}_u}&=&
\frac{\1{\tau=\infty}}{\inf_{u\in[0,\infty)}X_u}<\infty\,,
\end{eqnarray*}
where we used dominated convergence in the first equality, and on the other hand, from Lemma \ref{447},
$\lim_{n\rightarrow\infty}\left(X_{\tau^{(n)}_{1/h(n)}}-X^{(n)}_{\tau^{(n)}_{1/h(n)}}\right)\1{\tau^{(n)}=\infty}=0$. 
Then we have proved that $\lim_n\sigma_{1/h(n)}^{(n)}\1{\tau^{(n)}=\infty} = \zeta\1{\zeta<\infty}$, 
$\p_x$-a.s. 

Now using the equality
$\sigma_{1/h(n)}^{(n)} \1{\sigma^{(n)}_{1/h(n)}<\infty} = \sigma_{1/h(n)}^{(n)} \1{\tau^{(n)}=
\infty}-\sigma_{1/h(n)}^{(n)} \1{\tau^{(n)}_{1/h(n)}<\tau^{(n)}, \tau^{(n)}<\infty} $ together with the fact that 
$\lim_n\1{\tau^{(n)}_{1/h(n)}<\tau^{(n)}, \tau^{(n)}<\infty} = 0$, $\p_x$-a.s. allows us to obtain the convergence 
$\lim_n\sigma_{1/h(n)}^{(n)}\1{\sigma^{(n)}_{1/h(n)}<\infty} = \zeta\1{\zeta<\infty}$, 
$\p_x$-a.s. $\hfill\Box$\\

\begin{proof}[Proof of Proposition $\ref{0456}$] Let us first write,  
\begin{eqnarray}
&&\tilde{\zeta}^{(n)}\ind_{\{\tilde{\zeta}^{(n)}<\infty\}}-\int_0^{{\rm e}_n}\frac{du}{X_u}\ind_{\{\tau=\infty\}}
=\int_0^{{\rm e}_n}\frac{du}{X^{(n)}_u}\ind_{\{{\rm e}_n<\tau^{(n)}\}}-\int_0^{{\rm e}_n}\frac{du}{X_u}
\ind_{\{\tau=\infty\}}\nonumber\\
&&=\int_0^{{\rm e}_n}du\left(\frac{1}{X^{(n)}_u}-\frac{1}{X_u}\right)\ind_{\{{\rm e}_n<\tau^{(n)}\}}
+\int_0^{{\rm e}_n}\frac{du}{X_u}\left(\ind_{\{{\rm e}_n<\tau^{(n)}\}}-\ind_{\{\tau=\infty\}}\right)\,.\label{2300}
\end{eqnarray}
Then from (\ref{5440}), the expression (\ref{4200}) and monotone convergence, it suffices to prove that 
the above term tends almost surely to 0 as $n$ tends to $\infty$. The sequence $(\tau^{(n)})$ is non decreasing 
and tends almost surely to $\tau$. Moreover, for $\p$-almost every $\omega$, there is $n_0$ such that for all 
$n\ge n_0$, $\ind_{\{\tau^{(n)}=\infty\}}(\omega)=\ind_{\{\tau=\infty\}}(\omega)$ and since $({\rm e}_n)$ tends 
almost surely to $\infty$, for $\p$-almost every $\omega$, there is $n_1$, such that for all $n\ge n_1$, 
\[\ind_{\{{\rm e}_n<\tau^{(n)}\}}(\omega)=
\ind_{\{{\rm e}_n<\tau^{(n)}\}\cap\{\tau=\infty\}}(\omega)=\ind_{\{\tau=\infty\}}(\omega)\,,\]
from which we derive that 
\[\lim_n\int_0^{{\rm e}_n}\frac{du}{X_u}
\left(\ind_{\{{\rm e}_n<\tau^{(n)}\}}-\ind_{\{\tau=\infty\}}\right)=0\,,\;\;\mbox{$\p$-a.s.}\]
On the other hand, note the bound of the first term in (\ref{2300}),
\begin{align*}
\int_0^{{\rm e}_n}du\left(\frac{1}{X^{(n)}_u}-\frac{1}{X_u}\right)\ind_{\{{\rm e}_n<\tau^{(n)}\}}
\leq\frac{X_{{\rm e}_n}-X_{{\rm e}_n}^{(n)}}{\inf_{u\in[0,{\rm e}_n]}X_u^{(n)}} 
\int_0^{{\rm e}_n} \frac{du}{X_u}\1{\tau^{(n)}= \infty}du\,,\;\;\mbox{$\p$-a.s.}
\end{align*}
Then we conclude from Lemma \ref{447} in a similar way to the proof of Theorem \ref{1086}.
\end{proof}

\vspace*{.5in}

\noindent {\bf Acknowledgement} Both authors are  grateful for
financial support from the ANR project “Rawabranch” number ANR-23-CE40-0008.


\end{document}